\documentclass[11pt,reqno]{amsart}

\usepackage{amsmath,amsthm,amssymb,tocvsec2,mathrsfs,booktabs,tabularx}
\usepackage[tmargin=1.4in,bmargin=1.2in,rmargin=1.4in,lmargin=1.4in]{geometry}
\usepackage[breaklinks=true]{hyperref}
\usepackage[UKenglish]{babel}

\theoremstyle{plain}

\newtheorem{theorem}{Theorem}[section]
\newtheorem{corollary}[theorem]{Corollary}
\newtheorem{proposition}[theorem]{Proposition}
\newtheorem{lemma}[theorem]{Lemma}

\theoremstyle{definition}

\newtheorem{definition}[theorem]{Definition}
\newtheorem{notation}[theorem]{Notation}

\newtheorem{remark}[theorem]{Remark}

\numberwithin{equation}{section}
\numberwithin{table}{section}

\newcommand{\Beta}{\mathrm{B}}
\newcommand{\bb}{\mathbf{b}}
\newcommand{\bm}{\mathbf{m}}
\newcommand{\bu}{\mathbf{u}}
\newcommand{\bv}{\mathbf{v}}

\newcommand{\bx}{\mathbf{x}}
\newcommand{\by}{\mathbf{y}}
\newcommand{\bone}{\mathbf{1}}
\newcommand{\bzero}{\mathbf{0}}
\newcommand{\comp}{\mathrel{\circ}}
\newcommand{\dequals}{\stackrel{d}{=}}
\newcommand{\diag}{\mathop{\mathrm{diag}}}
\newcommand{\eps}{\varepsilon}
\newcommand{\expn}{\mathbb{E}}
\newcommand{\hx}{\hat{x}}
\newcommand{\hy}{\hat{y}}
\newcommand{\Id}{\mathrm{Id}}
\newcommand{\rd}{\mathrm{d}}
\newcommand{\std}{\,\rd}
\newcommand{\tr}{\mathop{\mathrm{tr}}}

\newcommand{\ds}{\displaystyle}
\newcommand{\tup}[1]{\textup{#1}}

\newcommand{\cB}{\mathcal{B}}
\newcommand{\cF}{\mathcal{F}}
\newcommand{\cL}{\mathcal{L}}
\newcommand{\cM}{\mathcal{M}}
\newcommand{\cN}{\mathcal{N}}
\newcommand{\cR}{\mathcal{R}}
\newcommand{\cS}{\mathcal{S}}

\newcommand{\C}{\mathbb{C}}
\newcommand{\Q}{\mathbb{Q}}
\newcommand{\R}{\mathbb{R}}
\newcommand{\Z}{\mathbb{Z}}
\newcommand{\N}{\mathbb{N}}

\newcommand{\arcos}{\mathop{\mathrm{arccos}}\nolimits}
\newcommand{\arcosh}{\mathop{\mathrm{arcosh}}\nolimits}
\newcommand{\cosech}{\mathop{\mathrm{cosech}}\nolimits}
\newcommand{\sech}{\mathop{\mathrm{sech}}\nolimits}

\expandafter\ifx\csname NC@find@L\endcsname\relax
 \newcolumntype{L}[1]{>{\raggedright\arraybackslash}p{#1}}
\fi

\subjclass[2020]{15B99 (primary); 
15A18, 
46C20, 
51M10, 
60G51 (secondary)} 

\keywords{Hyperbolic space, hyperboloid model, Lorentz--Gram matrix,
entrywise preserver, screw line, L\'evy measure,
L\'evy--Khintchine representation,
Bernstein function, conditionally negative kernel, Nevanlinna--Pick kernel,
fractional Ornstein--Uhlenbeck field, Bochner subordination}

\begin{document}

\title{Distance preservers for Lobachevsky space}

\author{Alexander Belton}
\address[A.~Belton]{School of Engineering, Computing and Mathematics,
University of Plymouth, Plymouth, UK}
\email{\tt alexander.belton@plymouth.ac.uk}

\author{Dominique Guillot}
\address[D.~Guillot]{University of Delaware, Newark, DE, USA}
\email{\tt dguillot@udel.edu}

\author{Apoorva Khare}
\address[A.~Khare]{Indian Institute of Science;
Analysis and Probability Research Group; Bangalore, India}
\email{\tt khare@iisc.ac.in}

\author{Mihai Putinar}
\address[M.~Putinar]{University of California at Santa Barbara, CA, USA} 
\email{\tt mputinar@math.ucsb.edu}

\date{21st September 2026}

\begin{abstract}
We obtain a complete description of the class of entrywise preservers of Lorentz--Gram matrices.
This resolves, for the case of constant negative curvature, the classification of entrywise preservers obtained
by Schoenberg in the zero-curvature (Euclidean) and constant-positive-curvature (spherical) settings.
These preservers admit a L\'evy--Khintchine-type representation and their
asymptotic characteristics are related to Krein's classification of
screw lines in Lobachevsky space. Connections with complete Nevanlinna--Pick
kernels and Bochner subordination are also obtained.
\end{abstract}

\maketitle

\tableofcontents

\section{Introduction}

The classification of functions that, when applied entrywise, preserve a distinguished class
of structured matrices has a long history, with contributions from Schur
towards the start of the twentieth century,
then proceeding via P\'olya and Szeg\H{o} to the seminal work of Schoenberg,
and more recently to entrywise preservers of moment matrices \cite{BGKP22} and of totally positive kernels \cite{BGKP23}.
Fundamentally, these results are all obtained by exploiting properties of positive definiteness and the many tools
available in this setting. While these classifications encompass Euclidean and spherical geometry
(zero and constant positive curvature), the corresponding results for hyperbolic geometry
(constant negative curvature) had been lacking until now.

\begin{table}[h]\label{tab:compact}
\centering
\small
\setlength{\tabcolsep}{5pt} 
\renewcommand{\arraystretch}{1.4}
\begin{tabular}{@{} L{0.175\textwidth} L{0.245\textwidth} L{0.245\textwidth} L{0.260\textwidth} @{}}
\toprule
 & \textbf{Euclidean} & \textbf{Spherical} & \textbf{Hyperbolic} \\
\midrule
Model  & $( \ell^2, d_2 )$ & $( S^\infty, d_\measuredangle ) \subseteq \ell^2$ & %
$( \cL, d_h ) \subseteq \ell^2$ \\[1ex]

Curvature & $0$ & $+1$ & $-1$ \\[2pt]

Gram invariant & $d_2( x, y )^2 = \| x - y \|^2$ & $\cos d_\measuredangle( x , y )= \langle x, y \rangle$ %
& $\cosh d_h( x, y ) = [ x, y ]$ \\[2pt]

Embedding certificate
 & $d^2$ conditionally negative definite
 & $\cos \comp d$ positive semidefinite
 & $\cosh \comp d$ has exactly one positive square \\[2pt]
 
Matrix class
 & squared distance: entries in~$\R_+$, zero diagonal, conditionally negative definite
 & correlation: entries in~$[ {-1}, 1 ]$, unit diagonal, positive semidefinite
 & Lorentz--Gram: entries in~$[ 1, \infty )$, unit diagonal, exactly one positive eigenvalue \\[2pt]

Continuous preserver $f$
 & Bernstein function vanishing at the origin
 & absolutely monotone function with $f( 1 ) = 1$
 & normalised Bernstein function in $\log t$ plus a regularised drift \\[2pt]
Representation for $f$
 & $\ds \beta t$ & $\ds \sum_{n = 0}^\infty c_n t^n$ & $\ds 1 + \beta \frac{t^\alpha - 1}{\alpha}$ \\[-12pt]
\raisebox{-8pt}{($\nu$ L\'evy measure)} & %
$\ds + \!\int_{( 0, \infty )} \!\!\bigl( 1 - e^{-t s} \bigr) \std\nu( s )$ & %
 \raisebox{-8pt}{($c_n \ge 0$, $\sum_n c_n = 1$)} & %
$\ds + \!\int_{( 0, \infty )} \!\!\bigl( 1 - t^{-s} \bigr) \std\nu( s )$ \\[12pt]

Metric transform $\varphi$ & $s \mapsto \sqrt{f( s^2 )}$ & $s \mapsto  \arcos\bigl( f( \cos s ) \bigr)$ & %
$s \mapsto \arcosh\bigl( f( \cosh s ) \bigr)$ \\[2pt]

Established by & Schoenberg, 1938 \cite{Sch38a,Sch38b} & Schoenberg, 1942 \cite{Sch42} & This article \\[2pt]

\bottomrule
\end{tabular}\vspace{2pt}%
\caption{Continuous distance preservers in constant-curvature spaces.
In each case, the transform $\varphi \comp d$ is again a metric of the same type
when $f$ acts entrywise to preserve the corresponding
class of Gram-type matrices.}
\end{table}

Here, we establish a complete classification for preservers of Gram-type matrices with Lorentzian signature.
Matrices with entries in $[ 1, \infty )$ that have leading diagonal
with all entries equal to $1$ and exactly one positive eigenvalue, counted with
multiplicity, are precisely the Gram matrices for infinite-dimensional
Lobachevsky space $\cL$; these matrices were named \emph{Lorentz--Gram}
by Loewner \cite{Loew65}.

The entrywise transformations preserving all Lorentz–Gram matrices are
shown herein to form a rigid class related to Bernstein functions. By doing this, we
see that the classical
Schoenberg classification for Euclidean and spherical distance geometry has a precise
negative-curvature counterpart.

We show that a function $f : [ 1, \infty ) \to [ 1, \infty )$ sends the collection
of Lorentz--Gram matrices to itself when applied entrywise if and only if
$f$ has the representation
\begin{equation}\label{rep}
f( t ) = 1 + \gamma 1_{( 1, \infty )}( t ) + \beta \frac{t^\alpha - 1}{\alpha} + %
\int_{( 0, \infty )} ( 1 - t^{-s} ) \std\nu( s ) \qquad ( t \ge 1 ),
\end{equation}
with the convention
\[
\frac{t^\alpha - 1}{\alpha} := \log t \qquad \text{if } \alpha = 0,
\]
where the \emph{order} $\alpha \in [ 0, 1 ]$, the \emph{drift} $\beta \in \R_+$,
the \emph{jump} $\gamma \in \R_+$
and the \emph{L\'evy measure}~$\nu$ is a positive Radon measure such that
\begin{equation}\label{keyineq}
\beta \ge \alpha ( 1 + \gamma + \nu\bigl( ( 0, \infty ) \bigr)  ) %
\qquad \text{if } \alpha \in ( 0, 1 ],
\end{equation}
so that $\nu$ is necessarily finite in this case, and
\[
\int_{( 0, \infty )} \min\{ s, 1 \} \std\nu ( s ) < \infty \qquad \text{if } \alpha = 0.
\]
This is the integrability condition for the L\'evy--Khintchine representation
of Bernstein functions,  which is no coincidence: we may re-write (\ref{rep}) as
\[
f( t ) = 1 + \gamma 1_{( 1, \infty )}( t ) + B( \log t ) + %
B'( \infty ) \Bigl( \frac{t^\alpha - 1}{\alpha} - \log t \Bigr) \qquad ( t \ge 1 ),
\]
where $B$ is a Bernstein function that vanishes at the origin:
\[
B( x ) = \beta x + \int_{( 0, \infty )} ( 1 - e^{-x s} ) \std\nu( s )  %
\qquad ( x \ge 0 ).
\]
The validity of the inequality (\ref{keyineq}) is equivalent to
$t \mapsto t^{-\alpha} f( t )$ being non-decreasing on $( 1, \infty )$.
In terms of the Bernstein-function representation, this requirement becomes
\[
B'( x ) \geq \alpha ( 1 + \gamma + B( x ) - B'( \infty ) x ) \qquad ( x > 0 ).
\]

Every function $f$ of this form, other than the constant function $f \equiv 1$,
is a \emph{distance preserver} in the following sense:
if $d$ is the hyperbolic metric on $\cL$ then $\varphi \comp d$ is also a metric on $\cL$, where
$\varphi := \arcosh \comp f \comp \cosh$. The naturalness of this relationship is explained below;
see Theorem~\ref{preservers}.

To obtain the representation (\ref{rep}), we first show that every entrywise
preserver has two characteristics governing its growth, which we call its order and type.
We then show by elementary methods that the following functions are preservers
for suitable choices of the parameters:
\[
t \mapsto t^a, \qquad t \mapsto 1 + b \log t, \qquad t \mapsto 1 + p - p t^{-a}, \qquad t \mapsto 1 - q + q t.
\]
These examples allow us to conjecture a candidate class
of preservers, which we prove is stable under the natural closure operations. Finally,
in Theorem~\ref{thm:entrywise}, we show that no other preservers exist.

Our proofs exploit the classification of screw lines in Lobachevsky space obtained by Krein \cite{Krein}. Their geometry, through the elliptic, parabolic and hyperbolic trichotomy obtained by Krein, corresponds to the analytic classification of bounded preservers
of zero order, unbounded preservers of zero order and preservers of positive order, respectively.

Bavaud has highlighted the relevance of entrywise preservers of Euclidean squared-distance matrices
to machine learning and data analysis  \cite{Bavaud}. Moreover, the groundbreaking article
by Krioukov and collaborators \cite{Krioukov} opened up a vast range of applications
of hyperbolic distance geometry; here we mention just a few \cite{Ganea, Nickel-Kiela, Sala}.
From this applied perspective, our study of inner transforms of hyperbolic distance is natural and timely.

In particular, Lorentz--Gram matrices are also the Gram matrices associated 
with kernels of real hyperbolic type \cite{MonodPy}. Thus our classification is also
the classification of functions which operate on kernels of real hyperbolic type.
In particular, the power-preservation result
obtained by Monod and Py \cite[Theorem~3.10]{MonodPy}, that if~$\alpha \in [ 0, 1 ]$
then $K^\alpha$ is a kernel of real hyperbolic type whenever $K$ is,
appears as a special case of the classification obtained here.

As consequences of our analysis, we are able to link the results obtained below
with the theories of complete Nevanlinna--Pick kernels (Section~\ref{psdkernels})
and Bochner subordination (Section~\ref{sec:probab}).
Furthermore, we resolve a question of Cohen and Lifshits (Proposition~\ref{prp:candl})
concerning fractional Ornstein--Uhlenbeck fields on hyperbolic space.

To ensure our presentation is self-contained and accessible to the broad audience we hope it to attract,
we provide a detailed introduction to Lobachevsky space in Section~\ref{sec:prelim}
and full proofs of auxiliary results in Appendix~\ref{appendix}.

We conclude in Section~\ref{sec:conclusion} with some reflections on what we
have achieved and what might be next.

\subsection{Notation}

We use standard mathematical notation, so that $\C$, $\R$ and $\Q$ denote the fields of complex, real and rational numbers. We let $\N$ denote the set of natural numbers, starting at $1$, whereas $\R_+$
and $\Z_+$ denote the sets of
non-negative real numbers and integers, respectively, with each including $0$.
A vector $\bv \in \R^n$ is a column vector by default, with 
the corresponding row vector being~$\bv^T$; the vector $\bone_n$ has $n$ entries, each equal to $1$.
The $n \times n$ identity matrix for the usual matrix multiplication is denoted~$\Id_{n \times n}$.

\subsection{Acknowledgements}
The authors thank the Mathematisches Forschungsinstitut Oberwolfach
for support and hospitality via the
\emph{Equivariant-kernel transforms and Schur functions} project
within its Research Fellows programme, where work on this paper was initiated.
AB thanks the Isaac Newton Institute for Mathematical Sciences,
Cambridge, for the support and hospitality of its Retreats programme,
where further work on this paper was undertaken via EPSRC grant EP/Z000580/1.
DG was partially supported by NSF grant \#2350067.
AK acknowledges support from the Government of India via an ARG grant (ANRF/ARG/2025/011665/MS) from ANRF, a Shanti Swarup Bhatnagar Award from CSIR and the DST FIST program-2021 TPN-700661.
Mihai Putinar was partially supported by a Simons Foundation collaboration grant.
Claude Opus 4.8, Opus 5 and Fable 5 were used to support the development
of some of the working in Sections 3 to 7 and the appendix.
The authors have verified the content of all AI-assisted material,
including all references and mathematical claims,
and take full responsibility for the content of this paper.

\section{Preliminaries}\label{sec:prelim}

Let $\ell^2 = \ell^2( \Z_+ )$ denote the real Hilbert space of square-summable sequences indexed by the non-negative integers,
with the standard inner product
\[
\langle x, y \rangle := \sum_{i = 0}^\infty x_i y_i %
\qquad \textrm{for } x = ( x_0, x_1, \ldots ) \textrm{ and } y = ( y_0, y_1, \ldots ).
\]
We introduce a symmetric bilinear form, the \emph{Lorentz form}:
\[
\ell^2 \times \ell^2 \to \R; \ %
( x, y ) \mapsto [ x, y ] := x_0 y_0 - \sum_{j = 1}^\infty x_j y_j = %
x_0 y_0 - \langle x', y' \rangle,
\]
where we write $x = ( x_0, x' ) \in \R \oplus \ell^2( \N )$;
we will use this decomposition without further comment
and will also consider $\ell^2( \N )$ as the subspace $\{ 0 \} \oplus \ell^2( \N )$
of $\ell^2$.
The Cauchy--Schwarz inequality gives that
\[
[ x, y ]^2 \leq \langle x, x \rangle \, \langle y, y \rangle \qquad \textrm{ for any } x, y \in \ell^2.
\]
In particular, the Lorentz form is jointly continuous
when $\ell^2$ is equipped with the norm topology.

We have also the \emph{hyperbolic Cauchy--Schwartz} inequality
\begin{equation}\label{ineq:qf}
[ x, y ]^2 \geq [ x, x ] \, [ y, y ] %
\qquad \textrm{for any } x, y \in \{ z \in \ell^2 : [ z, z ] > 0 \};
\end{equation}
to see this, note that $x_0 \neq 0$ and the quadratic polynomial
$p( t ) := [ t x + y, t x + y ]$ has a real root, since
\[
p( 0 ) = [ y, y ] > 0 \qquad \textrm{ and } p( -y_0 / x_0 ) = -\langle t x' + y' , t x' + y' \rangle \leq 0.
\]
We now specialise to the \emph{Lobachevsky space}
\begin{equation}
\cL := \{ x \in \ell^2 : x_0 > 0 \textrm{ and } [ x, x ] = 1 \}.
\end{equation}
It is immediate that if $x \in \cL$ then $x_0 \geq 1$, with equality exactly when $x = e := ( 1, 0 )$.

In fact, any $x \in \cL \setminus \{ e \}$ may be written uniquely in the form
\[
x = e \cosh \theta + \omega \sinh \theta,
\]
where $\theta \in ( 0, \infty )$ and $\omega \in \ell^2( \N )$ is a unit vector.
Conversely, the vector
\[
x = e \cosh \theta + \omega \sinh \theta \in \cL
\]
for any $\theta \in \R_+$ and any unit vector $\omega \in \ell^2( \N )$.
For the proof of both of these facts, see Proposition~\ref{decomposition}.

In particular, there is a path
\[
[ 0, 1 ] \to \cL; \ t \mapsto e \cosh( t \alpha ) + \omega \sinh( t \alpha )
\]
from $e$ to any $x \in \cL$, which is therefore path connected, so connected,
when $\cL$ has the subspace topology given by the Hilbert-space norm.
It now follows from (\ref{ineq:qf}) that
\[
[ x, y ] \geq 1 \qquad \text{for any } x, y \in \cL,
\]
since $[ x, x ] = 1$ and $z \mapsto [ x, z ]$ is continuous on $\cL$.

\subsection{The hyperbolic metric}

The \emph{hyperbolic distance} $r$ between points $x$ and $y$ in $\cL$
is defined by the identity
\[
[ x, y ] = \cosh r.
\]
Symmetry of the resulting metric is immediate, as is the fact that a point in $\cL$
is at zero distance from itself.
Now suppose $x$, $y \in \cL$ are such that
\[
1 = [ x, y ] = x_0 y_0 - \langle x', y' \rangle
\]
and let $\hx := ( 1, x' ) \in \ell^2$ and similarly for $\hy$. Then
\[
1 = [ x, x ] = x_0^2 - \langle x', x' \rangle \iff x_0^2 = %
1 + \langle x', x' \rangle = \langle \hx, \hx \rangle
\]
and similarly for $y_0$, so
\[
\langle \hx, \hx \rangle \, \langle \hy, \hy \rangle = %
x_0^2 y_0^2 = \langle \hx, \hy \rangle^2.
\]
Equality holds in the Cauchy--Schwarz inequality,
so $\hx$ and $\hy$ are linearly dependent,
and since $\hx_0 = \hy_0 = 1$, we have that $\hx = \hy$ and hence $x' = y'$.
Finally, we see that
\[
x_0^2 = \langle \hx, \hx \rangle = \langle \hy, \hy \rangle = y_0^2,
\] 
so $x_0 = y_0$ and therefore $x = y$. This shows that the metric is faithful.

To establish the triangle inequality, we first recall the existence of a useful
collection of isometries.
Given any $x \in \cL$, we have that
\[
[ x - e, x - e ] = [ x, x ] - 2 [ x, e ] + [ e, e ] = 1 - 2 x_0 + 1 = 2 ( 1 - x_0 ).
\]
Thus, if $x \neq e$, so that $x_0 > 1$, the vector	
\[
u := \frac{x - e}{\sqrt{2 ( x_0 - 1 )}} \in \ell^2 
\]
is such that $[ u, u ] = -1$. Hence the linear map
\[
R_x : \ell^2 \to \ell^2; \ y \mapsto y + 2 [ u, y ] u
\]
is such that
\[
R_x u = -u, \qquad R_x e = x, \qquad R_x x = e \quad \textrm{and} \quad %
R_x( R_x y ) = y \quad \textrm{for any } y \in \ell^2.
\]
Furthermore, if $y$, $z \in \ell^2$ then
\begin{align*}
[ R_x y, R_x z ] & = [ y + 2 [ u, y ] u, z + 2 [ u, z ] u ] \\[1ex]
 & = [ y, z ] + 4 [ u, y ] [ u, z] ( 1 + [ u, u ] ) \\[1ex]
 & = [ y, z ],
\end{align*}
so $R_x$ is a bijection from $\cL$ to itself that exchanges $e$ and $x$ and preserves the hyperbolic distance.

Now let $x$, $y$, $z \in \cL$ and let $r$, $s$, $t \in \R_+$ be such that
\[
\cosh r = [ x, y ], \qquad \cosh s = [ y, z ] \quad \textrm{and} \quad \cosh t = [ z, x ].
\]
Let $R = R_y$ if $y \neq e$ and let $R$ be the identity map otherwise.
Then the triangle inequality
\begin{equation}\label{triangle}
r + s \geq t
\end{equation}
holds if and only if the same inequality holds with $( x, y, z )$ replaced by
$( R x, R y = e, R z )$, and so, without loss of generality,
we may assume that $y =e$.
Applying $\cosh$ to both sides of (\ref{triangle})
and using the addition formula, we must show that
\begin{align*}
\cosh r \cosh s + \sinh r \sinh s \geq \cosh t & \iff %
[ x, e ] \, [ e, z ] + \sinh r \sinh s \geq [ z, x ] \\[1ex]
 & \iff x_0 z_0 + \sinh r \sinh s \geq x_0 z_0 - \langle x', z' \rangle.
\end{align*}
Furthermore, we have that $\sinh u = \sqrt{\cosh^2 u - 1}$ for any $u \in \R_+$,
so
\[
\sinh r \sinh s = \sqrt{[ x, e ]^2 - 1} \sqrt{[ e, z ]^2 - 1} = %
\sqrt{( x_0^2 - 1 ) ( z_0^2 - 1 )}.
\]
Noting that $1 = [ x, x ] = x_0^2 - \langle x', x' \rangle$ and similarly for $z$,
we need to show that
\[
\sqrt{\langle x', x' \rangle \, \langle z', z' \rangle} \geq -\langle x', z' \rangle,
\]
but this is an immediate consequence of the Cauchy--Schwarz inequality.
Hence the triangle inequality holds for the hyperbolic distance and
\begin{equation}
d : \cL \times \cL \to \R_+; \ ( x, y ) \mapsto \arcosh [ x, y ]
\end{equation}
is a metric.

\subsection{Topology}

The fact that the norm topology on $\cL$
is the same as the topology generated by the hyperbolic metric
follows from the readily verified identity
\begin{equation}\label{key-id}
\langle x - y, x - y \rangle = 2 \bigl( ( x_0 - y_0 )^2 + [ x, y ] - 1 \bigr) \qquad %
\textrm{for any } x, y \in \cL.
\end{equation}
If $x( n ) \to x$ in norm then $x( n )_0 \to x_0$ and,
by (\ref{key-id}), we see that $[ x( n ), x ] \to 1$,
hence $x( n ) \to x$ in the hyperbolic topology.
Conversely, if $x( n ) \to x$ in the hyperbolic topology
then $[ x( n ), x ] \to 1$. Furthermore, we have that
\[
| \arcosh x( n )_0 - \arcosh x_0 | = | d( x( n ), e ) - d( x, e ) | \leq d( x( n ), x ) \to 0
\] 
and therefore $x( n ) \to x$ in the norm topology, again by (\ref{key-id}).

The identity (\ref{key-id}) may also be used to show that Lobachevsky space $\cL$
is complete for both the $\ell^2$ norm and the hyperbolic metric.
In fact, as we can write
\[
\cL = \{ x \in \ell^2 : x_0 \geq 1 \textrm{ and } [ x, x ] = 1 \},
\]
it's clear that $\cL$ is a norm-closed subset of $\ell^2$, so $\cL$ is complete
for the $\ell^2$ norm.
If the sequence $\bigl( x( n ) \bigr)_{n = 1}^\infty$ is Cauchy for the metric $d$
then so is $\bigl( x( n )_0 = \cosh d( x( n ), e ) \bigr)_{n = 1}^\infty$,
using the reverse triangle inequality as above,
and therefore $\bigl( x( n ) \bigr)_{n = 1}^\infty$ is Cauchy for the norm,
by (\ref{key-id}).
Thus there exists $x \in \cL$ with $x( n ) \to x$ in the norm topology.
In particular, we have that $x( n )_0 \to x_0$ and
therefore $[ x( n ), x ] \to 1$ by another application of~(\ref{key-id}).
This shows that $\cL$ is also complete for the hyperbolic metric.

\subsection{Lorentz--Gram matrices}

We are interested in the equivalent of Gram matrices for the bilinear form
$[ \cdot, \cdot ]$; Loewner named these \emph{Lorentz--Gram matrices}
in \cite{Loew65}. These will be characterised shortly: see Theorem~\ref{Krein-Gram}.

Let
\begin{align}
\cM & := \smash[b]{\bigcup_{n = 1}^\infty %
\{ M = [ m_{i j} ] \in [ 1, \infty )^{n \times n} : %
M \textrm{ is symmetric, } m_{i i} = 1 %
\textrm{ for any } i \}} \\[1ex]
\text{and} \qquad \nonumber \\[1ex]
\cM_{1+} & := %
\{ M \in \cM : M \textrm{ has exactly one positive eigenvalue} \}.
 \end{align}
 
Here and throughout, we count eigenvalues by multiplicity,
so \smash[b]{$\begin{bmatrix} 1 & 0 \\[1ex] 0 & 1 \end{bmatrix}$} does not have
exactly one positive eigenvalue.

\begin{lemma}\label{Loewner}
If $M$ and $N$ are real symmetric matrices and their difference $M - N$
is positive semidefinite then $M$ has at least as many positive eigenvalues as~$N$.
\end{lemma}
\begin{proof}
This is a consequence of Weyl's interlacing theorem: see \cite[Corollary~7.7.4(d)]{HJ}.
\end{proof}
 
\begin{theorem}\label{Krein-Gram}
The following are equivalent for a matrix $M = [ m_{i j} ]_{i, j = 1}^n \in \cM$.
\begin{itemize}
\item[(1)] The matrix $M \in \cM_{1+}$.
\item[(2)] There exists $\bv \in [ 1, \infty )^n$ and a positive semidefinite matrix $A \in \R^{n \times n}$ such that $M = \bv \bv^T - A$
and $A \bv = \bzero$.
\item[(3)] There exists $\bv \in [ 1, \infty )^n$ and a positive semidefinite matrix $A \in \R^{n \times n}$ such that $M = \bv \bv^T - A$.
\item[(4)] There exists $\bv \in \R^n$ and a positive semidefinite matrix $A \in \R^{n \times n}$ such that $M = \bv \bv^T - A$.
\item[(5)] There exist $x( 1 )$, \ldots, $x( n ) \in \cL$ such that $m_{i j} = [ x( i ), x( j ) ]$
for any $i$ and $j$.
\end{itemize}
\end{theorem}
\begin{proof}
We note first that $\tr M = n > 0$, so $M$ always has at least one positive eigenvalue.

If (4) holds then, by Lemma~\ref{Loewner},
the matrix~$M$ has no more than one positive eigenvalue and so (1) holds.

If (1) holds then, by the Perron--Frobenius theorem \cite[Theorem~8.2.8]{HJ}
and spectral decomposition for real symmetric matrices, we can write
$M = \bv \bv^T - A$, where the vector $\bv  = ( v_1, \ldots, v_n )^T$ has
positive entries, the matrix $A = [ a_{i j} ]_{i, j = 1}^n$ is positive semidefinite
and $A \bv = \bzero$.
Furthermore, it holds that $v_i^2 = m_{i i} + a_{i i} \geq 1$
and so $v_i \geq 1$. Hence (1) implies (2).

It is clear that (2) implies (3) and (3) implies (4), so the first four statements are equivalent.

Given (3), we can write $A = B^T B$ for some matrix $B$ with columns
$\bb_1$, \ldots, $\bb_n \in \R^n$. Setting
$x( i ) := ( v_i, \bb_i^T, 0, 0, \ldots ) \in \ell^2( \Z_+ )$, we
see that $[ x( i ), x( j ) ] = v_i v_j - \bb_i^T \bb_j = m_{i j}$ for any $i$ and $j$,
so $x(1 )$, \ldots, $x( n ) \in \cL$ and (3) implies (5).

Finally, if (5) holds then
\[
[ x( i ), x( j ) ] = x( i )_0 x( j )_0 - \langle x( i )', x( j )' \rangle = %
( \bu \bu^T - G )_{i j} \qquad \text{for any } i \textrm{ and } j,
\]
where $\bu := ( x( 1 )_0, \ldots, x( n )_0 )^T \in [ 1, \infty )^n$
and $G$ is a Gram matrix, so positive semidefinite. Thus (5) implies (3).
\end{proof}

We note that $\cM_{1+}$ is not closed under the Schur product: if
\[
M_{a, b} := %
\begin{bmatrix} 1 & a & b \\[1ex] a & 1 & a \\[1ex] b & a & 1 \end{bmatrix} %
\qquad ( a, b \ge 1 )
\]
then, by Proposition~\ref{examplematrix}, the matrix
$M_{2, 6} \in \cM_{1+}$ but $M_{2, 6}^{\circ 2} = M_{4, 36} \notin \cM_{1+}$, because
\[
2 ( 2 )^2 - 1 = 7 \ge 6 \qquad \text{and} \qquad 2 ( 4 )^2 - 1 = 31 < 36.
\]
The set $\cM_{1+}$ is also not convex: we have that $M_{1, 1} \in \cM_{1+}$ and $M_{2, 7} \in \cM_{1+}$
but $\cM_{3/2, 4} \not\in \cM_{1+}$, because
\[
2 ( 1 )^2 - 1 = 1 \ge 1, \qquad 2 ( 2 )^2 - 1 = 7 \ge 7 \quad \text{and} \quad %
2 ( 3 / 2 )^2 - 1 = 7 / 2 < 4.
\]

We note that $\cM_{1+}$ is closed under taking principal submatrices.
If $M \in \cM_{1+}$ has an $m \times m$ principal submatrix $A$,
with eigenvalues
\[
\lambda_1 \le \cdots \le \lambda_{n - 1} \le 0 < \lambda_n %
\qquad \text{and} \qquad %
\mu_1 \le \cdots \le \mu_m,
\]
respectively, then, by the Cauchy interlacing theorem \cite[Theorem~4.3.28]{HJ},
\[
\mu_1 \le \cdots \le \mu_{m - 1} \le \lambda_{n - m + m - 1} = \lambda_{n - 1} \le 0
\]
and $\mu_m > 0$ because $A$ has positive trace.

We also introduce relaxations of $\cM$ and $\cM_{1+}$, where
the leading diagonal is allowed to be non-constant:
\begin{align}
\cM^* & := %
\bigcup_{n = 1}^\infty \{ M = [ m_{i j} ]_{i, j = 1}^n \in [ 1, \infty )^{n \times n} : %
M \text{ is symmetric} \} \\[1ex]
\text{and} \quad \cM^*_{1+} & := %
\{ M \in \cM^* : M \text{ has exactly one positive eigenvalue} \}.
\end{align}

A real symmetric matrix with trace at least one
has exactly one positive eigenvalue if and only if
its second largest eigenvalue is non-positive. It follows that the sets
$\cM_{1+}$ and $\cM_{1+}^*$ are closed under pointwise limits.

We note also that $\cM_{1+}^*$ is closed under scaling: if $M \in \cM_{1+}^*$
then $\lambda M \in \cM_{1+}^*$ for any $\lambda \ge 1$.

\begin{remark}
In previous work \cite{BGKP26}, we described those functions that act entrywise to
preserve real symmetric matrices with exactly $k$ negative eigenvalues.
The restricted class of functions found there, namely the positive homotheties and,
when $k = 1$, the negative constants, may be compared to the far richer family
of preservers that appears here, where the entries of elements of $\cM_{1+}^*$
are constrained to lie in~$[ 1, \infty )$.
\end{remark}

\section{Distance preservers}

Our initial aim is to classify the collection of functions of the form
$\varphi : \R_+ \to \R_+$ such that $\varphi \comp d$ is a metric on $\cL$.
In particular, we must have that $\varphi( x ) = 0$ if and only if $x = 0$.
We will show the relationship between this property and the entrywise preservation
of~$\cM_{1+}$ and~$\cM_{1+}^*$ in Theorem~\ref{preservers}.

We begin our approach via the following result,
which was announced by Krein in~1948 \cite{Krein} and
appears as Theorem~6.1 in \cite{IK-2}. The great utility of this theorem
is that it replaces a non-linear geometrical question with a linear-algebraic
spectral condition, so bringing the problem within the framework of
entrywise matrix preservers.

\begin{theorem}[Krein, 1948]\label{Krein-embedding}
Let $( Q, \rho )$ be a set $Q$ equipped with a map
$\rho : Q \times Q \to \R_+$ such that
\[
\rho( p, q ) = \rho( q, p ) \qquad \textrm{and} \qquad \rho( p, p ) = 0 %
\qquad \textrm{for any } p, q \in Q.
\]
There exists a map $\Phi : Q \to \cL$ such that
\[
d\bigl( \Phi( p ), \Phi( q ) \bigr) = \rho( p, q ) %
\qquad \textrm{for any } p, q \in Q
\]
if and only if the real symmetric matrix
\[
\bigl[ \cosh \rho( q_i, q_j ) \bigr]_{i, j = 1}^n
\]
has exactly one positive eigenvalue for any $q_1$, \ldots, $q_n \in Q$.
\end{theorem}

The eigenvalue characterisation of Lorentz--Gram matrices
and Krein's certificate of isometric embedding into Lobachevsky space
were recently rediscovered by Tabaghi and Dokmani\'c \cite{Tabaghi}.

Given $\varphi$ as above, we let
\[
f := \cosh \comp \varphi \comp \arcosh : [ 1, \infty ) \to [ 1, \infty )
\]
and note that the condition that $\varphi( x ) = 0$ if and only if $x = 0$
is equivalent to requiring that $f( t ) = 1$ if and only if $t = 1$.
Furthermore, if $Q := \cL$ and $\rho := \varphi \comp d$ then
the initial hypotheses of Theorem~\ref{Krein-embedding} are satisfied, and
\[
\cosh \rho( p, q ) = ( \cosh \comp \varphi \comp d )( p, q ) = f\bigl( [ p, q ] \bigr) %
\qquad \textrm{for any } p, q \in Q = \cL.
\]
It follows from Theorem~\ref{Krein-embedding} that if the matrix
$[ f\bigl( [ x( i ), x( j ) ] \bigr) ]_{i, j =1}^n$ has exactly one positive eigenvalue
for any $x( 1 )$, \ldots, $x( n ) \in \cL$ then there exists a map
\[
\Phi : \cL \to \cL \quad \text{such that} \quad %
\varphi\bigl( d( x, y ) \bigr) = d\bigl( \Phi( x ), \Phi( y ) \bigr) %
\quad \text{for any } x, y \in \cL.
\]
In particular, if $x$, $y$, $z \in \cL$ then
\[
\varphi\bigl( d( x, y ) \bigr) + \varphi\bigl( d( y, z ) \bigr) = %
d\bigl( \Phi( x ), \Phi( y ) \bigr) + d\bigl( \Phi( y ), \Phi( z ) \bigr) \geq %
d\bigl( \Phi( x ), \Phi( z ) \bigr) = \varphi\bigl( d( x, z ) \bigr)
\]
and therefore $\varphi \comp d$ is a metric on $\cL$.

We can summarise this working in the following theorem. The equivalence
of (3) and (4) will be shown below: see Theorem~\ref{bigtestmatrix}.

We use the notation $f[ A ] := [ f( a_{i j} ) ]$
to denote the result of applying a function $f$ entrywise
to a matrix $A = [ a_{i j} ]$ and $f[ - ] := A \mapsto f[ A ]$.

\begin{theorem}\label{preservers}
Let $\varphi : \R_+ \to \R_+$ be such that
$\varphi^{-1}\bigl( \{ 0 \} \bigr) = \{ 0 \}$ and let
\[
f : [ 1, \infty ) \to [ 1, \infty ); \ t \mapsto \cosh \varphi( \arcosh t ).
\]
The following are equivalent.
\begin{itemize}
\item[(1)] There exists a map $\Phi : \cL \to \cL$ such that
\[
d\bigl( \Phi( x ), \Phi( y ) \bigr) = \varphi\bigl( d( x, y ) \bigr) %
\qquad \textrm{for any } x, y \in \cL.
\]
\item[(2)] The matrix
$[ f\bigl( [ x( i ), x( j ) ] \bigr) ]_{i, j = 1}^n \in \cM_{1+}$
for any $x( 1 )$, \ldots, $x( n ) \in \cL$.
\item[(3)] The map $f$ when applied entrywise preserves $\cM_{1+}$,
that is, $f[ \cM_{1+} ] \subseteq \cM_{1+}$.
\item[(4)] The map $f$ when applied entrywise preserves $\cM_{1+}^*$,
that is, $f[ \cM_{1+}^* ] \subseteq \cM_{1+}^*$.
\end{itemize}
Any of these four statements implies the following one.
\begin{itemize}
\item[(5)] The map $\varphi \comp d : \cL \times \cL \to \R_+$ is a metric on $\cL$.
\end{itemize}
\end{theorem}
The embedding $\Phi$ in Theorem~\ref{preservers}(1) is continuous
if $\varphi$ (equivalently,~$f$) is.

By Proposition~\ref{counterexample}, this leads to a strictly smaller class of functions
than initially considered: if $\vartheta( x ) \equiv \min\{ x, 1 \}$ then
$\vartheta \comp d$
is a metric on~$\cL$ but the corresponding map~$f$ does not preserve
the set $\cM_{1+}$ when acting entrywise.

Hence we modify our initial aim as follows:
we seek to classify the collection of functions of the form
$f : [ 1, \infty ) \to [ 1, \infty )$ such that $f( t ) = 1$ if and only if $t = 1$
and for which the symmetric kernel
\[
\cL \times \cL \to [ 1, \infty ); \ ( x, y ) \mapsto f\bigl( [ x, y ] \bigr)
\]
\emph{has exactly one positive square}: for any $x( 1 )$, \ldots, $x( n ) \in \cL$,
the matrix
\begin{equation}\label{fM}
f[ M ] := [ f\bigl( [ x( i ), x( j ) ] \bigr) ]_{i, j = 1}^n
\end{equation}
has exactly one positive eigenvalue. By a slight abuse of terminology,
we will call such a function $f$ a \emph{distance preserver for Lobachevsky space},
or just a \emph{distance preserver} for the sake of brevity.

Thus $f$ is a distance preserver if and only if
$f^{-1}\bigl( \{ 1 \} \bigr) = \{ 1 \}$ and $f[ - ]$
sends the set~$\cM_{1+}$ to itself. In particular, we see that
the class of distance preservers is closed under composition.

\begin{remark}
A symmetric kernel $K : X \times X \to [ 1, \infty )$  with $K( x, x ) = 1$
for any~$x \in X$ is said to be \emph{of real hyperbolic type} \cite{MonodPy, Monod}
if and only if it admits a representation
\[
K( x, y ) = \cosh d\bigl( \Phi( x ), \Phi( y ) \bigr) \qquad ( x, y \in X )
\]
for some $\Phi : X \to \cL$. Equivalently, every matrix $\bigl[ K( x_i, x_j ) \bigr]_{i, j = 1}^n$ has exactly one positive eigenvalue. Thus the problem considered here is exactly the classification of every function $f$ such that $f \comp K$ is of real hyperbolic type
whenever $K$ is.
\end{remark}

It follows directly from Theorem~\ref{Krein-Gram}
that the Lorentz form $[ \cdot, \cdot ]$ has
exactly one positive square, and so the identity function is a distance preserver
(as it should be).

It also follows directly from Theorem~\ref{Krein-Gram} that if,
for arbitrary $x( 1 )$, \ldots, $x( n ) \in \cL$, the matrix $f[ M ]$
in (\ref{fM}),  can be written in the form $\bv \bv^T - A$,
where the vector $\bv \in \R^n$ and the matrix $A$ is positive semidefinite,
then $f$ is a distance preserver.

We now complete the proof of Theorem~\ref{preservers}.

\begin{theorem}\label{bigtestmatrix}
Given $R = [ r_{i j} ]_{i, j = 1}^k \in \cM^*$ and $m \in \N$, we let $n := k m$
and define the blow-up block matrix $M = [ m_{p q} ]_{p, q = 1}^n \in \cM$
by setting
\[
m_{p q} := \left\{ \begin{array}{ll}
 1 & \text{if } p = q, \\[1ex]
 r_{i i} & \text{if } p \neq q, \ p, q \in C_i, \\[1ex]
 r_{i j} & \text{if } p \in C_i, \ q \in C_j, \ i \neq j,
\end{array}\right.
\]
where $C_i := \{ ( i - 1 ) m + 1, \ldots, i m \}$ for $i = 1$, \ldots, $k$.
If $R$ has at most one positive eigenvalue then $M \in \cM_{1+}$.

It follows that if $f : [ 1, \infty ) \to [ 1, \infty )$ with $f( 1 ) = 1$ then
$f[ - ]$ preserves $\cM_{1+}$ if and only if $f[ - ]$ preserves $\cM^*_{1+}$.
\end{theorem}
\begin{proof}
Given $i \in \{ 1, \ldots, k \}$, we note that if $\bu \in \R^n$ is supported on $C_i$,
so that $u_p = 0$ whenever $p \not\in C_i$,
and $\bone_{C_i}^T \bu = 0$ then $M \bu = ( 1 - r_{i i} ) \bu$.
This shows that $1 - r_{i i}$ is an eigenvalue of $M$ with multiplicity at least $m - 1$
and $k ( m - 1 )$ of the eigenvalues of $M$ are non-positive.

Now for each $i \in \{ 1, \ldots, k \}$, let $\bv_i = m^{-1 / 2} \bone_{C_i}$
and note that the set $V = \{ \bv_1, \ldots, \bv_k \}$ is orthogonal
to any vector $\bu$ considered in the previous paragraph.
Furthermore, we have that
\[
\bv_i^T M \bv_i = 1 + ( m - 1 ) r_{i i} =: q_{i i} %
\quad \text{and} \quad %
\bv_i^T M \bv_j = m r_{i j} =: q_{i j} \qquad ( i \neq j ).
\]
Thus on the space spanned by $V$ the matrix $M$ is unitarily equivalent to
$Q = [ q_{i j} ]_{i, j =1}^k$.
By Sylvester's law of inertia, the matrix $Q$ has the same inertia as
\[
D Q D = %
[ \delta_{i j} \bigl( r_{i i} + m^{-1} ( 1 - r_{i i} ) \bigr) + %
( 1 - \delta_{i j} ) r_{ i j} ]_{i, j = 1}^k,
\]
where $D := \diag( m^{-1 / 2}, \ldots, m^{-1 / 2} )$.
In particular, the matrix $M$ has the same number of positive eigenvalues as $D Q D$.
We note that
\[
R - D Q D = \diag\bigl( m^{-1} ( r_{1 1} - 1 ), \ldots, m^{-1} ( r_{k k} - 1 ) \bigr),
\]
which is positive semidefinite, and so $R$ has at least as many positive eigenvalues as~$D Q D$ does, by Lemma~\ref{Loewner}, and so at least as many as $M$ does.
Since $M$ has at least one positive eigenvalue, because its trace is positive,
the first claim follows.

For the second, we assume $f[ - ]$ preserves $\cM_{1+}$
and suppose $R \in \cM_{1+}^*$. The previous result gives that $M \in \cM_{1+}$
and so $f[ M ] \in \cM_{1+}$.
We note that $f[ M ]$ has the same form as $M$, with $r_{i j}$ replaced by $f( r_{i j} )$
for all $i$ and $j$. Hence the matrix
\[
[ \delta_{i j} ( f( r_{i i} ) +  m^{-1} \bigl( 1 - f( r_{i i} ) \bigr) ) + %
( 1 - \delta_{i j} ) f( r_{i j} ) ]_{i, j = 1}^k
\]
has exactly one positive eigenvalue, so letting $m \to \infty$
gives the same for $f[ R ]$. The converse is immediate.
\end{proof}

\begin{corollary}\label{cty}
A distance preserver $f : [ 1, \infty ) \to [ 1, \infty )$ is non-decreasing
on $[ 1, \infty )$ and continuous on $( 1, \infty )$.
\end{corollary}
\begin{proof}
For such a function $f$ and any $s$, $t \in [ 1, \infty )$ we see that
\[
\begin{bmatrix} s & \sqrt{s t} \\[1ex] \sqrt{s t} & t \end{bmatrix} \in \cM_{1+}^* %
\implies %
\begin{bmatrix} f( s ) & f\bigl( \sqrt{s t} \bigr) \\[1ex]
 f\bigl( \sqrt{s t} \bigr) & f( t ) \end{bmatrix} \in \cM_{1+}^* %
 \iff \sqrt{f( s ) f( t )} \le f( \sqrt{s t} ).
\]
Furthermore, we may assume $s \le t$ and then
\[
\begin{bmatrix} s & t \\[1ex] t & s \end{bmatrix} \in \cM_{1+}^* %
\implies %
\begin{bmatrix} f( s ) & f( t ) \\[1ex]
 f( t ) & f( s ) \end{bmatrix} \in \cM_{1+}^* \iff f( s ) \le f( t ).
\]
Hence $f$ is multiplicatively midpoint concave and non-decreasing. It follows
that $f$ is continuous on $( 1, \infty )$ \cite[Theorem~71A]{RV}.
\end{proof}

\begin{definition}
Given a distance preserver $f$, its \emph{continuous regularisation} $\tilde{f}$
is defined by setting $\gamma := f( {1+} ) - 1 \ge 0$,
which exists as $f$ is non-decreasing and bounded below, and then
\[
\tilde{f} : [ 1, \infty ) \to [ 1, \infty ); \ %
t \mapsto ( 1 + \gamma ) 1_{\{1 \}} ( t ) + 1_{( 1, \infty )}( t ) f( t ).
\]
We note that $\tilde{f}$ is continuous and non-decreasing.
\end{definition}

We claim that $\tilde{f}[ - ]$ maps $\cM_{1+}$ into $\cM_{1+}^*$.
To see this, let $M = \bv \bv^T - A \in \cM_{1+}$, where $A$ is positive semidefinite
and $\bv \in [ 1, \infty )^n$. For any $\eps > 0$, let
\[
M_\eps := M + \eps \bv \bv^T = \bu \bu^T - A, \qquad %
\text{where } \bu := \sqrt{1 + \eps} \, \bv.
\]
Then $M_\eps \in \cM_{1+}^*$, by Lemma~\ref{Loewner} and the fact
that $M_\eps$ has positive trace. Furthermore, we have that
$M_\eps\in ( 1, \infty )^{n \times n}$
and so $\tilde{f}[ M_\eps ] = f[ M_\eps ] \in \cM_{1+}^*$,
by Theorem~\ref{preservers}(4). Letting $\eps \to 0+$, we see that
$\tilde{f}[ M_\eps ] \to \tilde{f}[ M ]$ entrywise and $\tilde{f}[ M ] \in \cM_{1+}^*$,
as claimed.

This observation will allow us to reduce the classification of distance preservers
to the classification of the subclass of distance preservers that are continuous.

\subsection{Two characteristics}

In this section we show that every distance preserver has two fundamental characteristics,
which we call \emph{order} and \emph{type} in analogy with the use of these terms in complex-function theory. They provide the basis for the classification that we are developing.

Our first result provides a restriction on the growth rate of any distance preserver.

\begin{proposition}\label{prp:nondec}
Given any distance preserver $f$, the function
\[
( 1, \infty ) \to ( 0, \infty ); \ t \mapsto \frac{f( t ) - 1}{t - 1}
\]
is non-increasing. 
\end{proposition}
\begin{proof}
For any $\lambda \in ( 0, 1 ]$, let
\begin{equation}\label{phidef}
\theta_\lambda : [ 1, \infty ) \to [ 1, \infty ); \ t \mapsto \lambda t + 1 - \lambda
\end{equation}
and note that $\theta_\lambda$ is strictly increasing, $\theta_\lambda( 1 ) = 1$
and $\theta_\lambda \circ \theta_\mu = \theta_{\lambda \mu}$
for every choice of~$\lambda$ and $\mu$.

We know from Proposition~\ref{example2matrix} that,
given an integer $m \ge 2$ and $a$, $b$, $c \ge 1$, the matrix
\[
N^{(m)}_{a, b, c } \in \cM_{1+} \quad \textrm{if and only if } \quad %
b^2 \geq \theta_{( m - 1 ) / m}( a ) \theta_{( m - 1 ) / m}( c ).
\]
Since we have that $f[ N^{(m)}_{a, b, c} ]= N^{(m)}_{f( a ), f( b ),  f( c )}$,
if $f$ is a distance preserver then
\[
f\bigl( \sqrt{\theta_\lambda( a ) \theta_\lambda( c )} \bigr) \geq %
\sqrt{\theta_\lambda\bigl( f( a ) \bigr) \theta_\lambda( f( c ) \bigr)} %
\qquad ( a, c \geq 1, \ \lambda = ( m - 1 ) / m ).
\]
Letting $m \to \infty$ recovers multiplicative midpoint concavity once again:
\[
f( \sqrt{a c } ) \ge \sqrt{f ( a ) f( c )} \qquad ( a, c \geq 1 ).
\]
Taking $a = c$ gives that
\begin{equation}\label{ineq}
f\bigl( \theta_\lambda( a ) \bigr) \ge \theta_\lambda\bigl( f( a ) \bigr) %
\qquad ( a \ge 1 )
\end{equation}
and so
\[
f\bigl( \theta_{\lambda \mu}( a ) \bigr) = %
f( \theta_\lambda\bigl( \theta_\mu( a ) \bigr) ) \ge %
\theta_\lambda( f\bigl( \theta_\mu( a ) \bigr) ) \ge %
\theta_\lambda( \theta_\mu\bigl( f( a ) \bigr) ) = %
\theta_{\lambda \mu}\bigl( f( a ) \bigr)
\]
for $\lambda = ( m - 1 ) / m$ and $\mu = ( n - 1 ) / n$. Since
\[
\frac{M}{N} = \prod_{j = M + 1}^N \frac{j - 1}{j} \qquad ( M, N \in \N, \ M < N ),
\]
we see that the inequality (\ref{ineq}) holds for all $\lambda \in \Q \cap ( 0, 1 ]$
and so for all $\lambda \in ( 0, 1 ]$ by continuity. [For $a = 1$ the inequality is
trivial and $f$ is continuous on $( 1, \infty )$.]

If $t > s > 1$ then we can write $s = \lambda t + 1 - \lambda$
with $\lambda = ( s - 1 ) / ( t - 1 )$ and (\ref{ineq}) gives that
\begin{align*}
f\bigl( \theta_\lambda( t ) \bigr) \ge \theta_\lambda\bigl( f( t ) \bigr) & \iff
f( \lambda t + 1 - \lambda ) \ge \lambda f( t ) + 1 - \lambda \\[1ex]
 & \iff ( t - 1 ) f( s ) \ge ( s - 1 ) f( t ) + t - 1 - s + 1 \\[1ex]
 & \iff \frac{f( s ) - 1}{s - 1} \ge \frac{f( t ) - 1}{t - 1}.
\end{align*}
This gives the result.
\end{proof}

We now associate two characteristics to any distance preserver
that quantify its growth rate.

\begin{theorem}
Given any distance preserver $f$, the following limits exist:
\begin{equation}
\alpha := \lim_{t \to \infty} \frac{\log f( t )}{\log t} \in [ 0, 1 ] %
\qquad \text{and} \qquad %
C := \lim_{t \to \infty} t^{-\alpha} f( t ) \in [ f( 1+ ), \infty ].
\end{equation}
These are respectively the \emph{order} and the \emph{type} of $f$.
\end{theorem}
\begin{proof}
We know from the proof of Corollary~\ref{cty} that
\[
F : ( 0, \infty ) \to ( 0, \infty ); \ x \mapsto \log f( e^x )
\]
is concave and non-decreasing, so
\[
F( t x + ( 1 - t ) x_0 ) \ge t F( x ) + ( 1 - t ) F( x_0 ) %
\qquad ( t \in [ 0, 1 ], \ x > x_0 > 0 )
\]
and, taking $t = ( y - x_0 ) / ( x - x_0 )$ for $y \in ( x_0, x )$,
\begin{equation}\label{chord}
\frac{F( y ) - F( x_0 )}{y - x_0} \ge \frac{F( x ) - F( x_0 )}{x - x_0} \ge 0.
\end{equation}
Hence
\[
\alpha := \lim_{t \to \infty} \frac{\log f( t )}{\log t} = %
\lim_{x \to \infty} \frac{F( x )}{x} = %
\lim_{x \to \infty} \frac{F( x ) - F( x_0 )}{x - x_0} \frac{x - x_0}{x} = %
\lim_{x \to \infty} \frac{F( x ) - F( x_0 )}{x - x_0}
\]
exists and lies in $\R_+$.

It follows immediately from Proposition~\ref{prp:nondec} that if $t > 2$ then
\begin{equation}\label{growthest}
f( t ) \le 1 + ( f( 2 ) - 1 ) ( t - 1 ) = O( t ) \qquad \text{as } t \to \infty.
\end{equation}
Hence there exists $c > 0$
such that $f( t ) \le c t$ for all sufficiently large $t$, whence
\[
\alpha \le \lim_{t \to \infty} \frac{\log c + \log t}{\log t} = 1.
\]
Furthermore, if $t = e^x$ and $t_0 = e^{x_0}$, with $t > t_0$, then
\[
\frac{F( x ) - F( x_0 )}{x - x_0} \ge \alpha \iff %
\log\bigl( f( t ) / f( t_0 ) \bigr) \ge \alpha \log( t / t_0 ) \iff %
t^{-\alpha} f( t ) \ge t_0^{-\alpha} f( t_0 ),
\]
so $t \mapsto t^{-\alpha} f( t )$ is non-decreasing on $( 1, \infty )$.
As $f$ is non-decreasing on $[ 1, \infty )$, we see that
\[
\lim_{t \to 1+} t^{-\alpha} f( t ) = f( 1+ ) \ge 1,
\]
so $C := \lim_{t \to \infty} t^{-\alpha} f( t ) \in [ f( 1+ ), \infty ]$
and $f( t ) \ge f( 1+ ) t^\alpha$ for all $t > 1$.
\end{proof}

\subsection{Four examples}\label{examples}

We now show through elementary arguments the existence of four classes of
distance preservers. They will be used to motivate our working and illustrate our results.

If $x \in \cL$ then
\[
1 = [ x, x ] = x_0^2 - \langle x', x' \rangle \implies%
\langle x', x' \rangle < 1 + \langle x', x' \rangle = x_0^2.
\]
Hence if $x$, $y \in \cL$ then $| \langle x', y' \rangle| < x_0 y_0$,
by the Cauchy--Schwarz inequality, so writing
\[
[ x, y ] = x_0 y_0 \Bigl( 1 - \frac{\langle x', y' \rangle}{x_0 y_0} \Bigr)
\]
gives the series expansions
\[
[ x, y ]^a = %
x_0^a y_0^a \Bigl( 1 - a \frac{\langle x', y' \rangle}{x_0 y_0} + %
\frac{a ( a - 1 )}{2} \Bigl( \frac{\langle x', y' \rangle}{x_0 y_0}  \Bigr)^2 - %
\frac{a ( a - 1 ) ( a - 2 )}{6} %
\Bigl( \frac{\langle x', y' \rangle}{x_0 y_0}  \Bigr)^3 + \cdots \Bigl)
\]
and
\[
\log [ x, y ] = \log x_0 + \log y_0 - \frac{\langle x', y' \rangle}{x_0 y_0} - %
\frac{1}{2} \Bigl( \frac{\langle x', y' \rangle}{x_0 y_0} \Bigr)^2 - %
\frac{1}{3} \Bigl( \frac{\langle x', y' \rangle}{x_0 y_0} \Bigr)^3 - \cdots.
\]
Thus if $a \in ( 0, 1 ]$ and
\begin{equation}
f_1 : [ 1, \infty ) \to [ 1, \infty ); \ t \mapsto t^a
\end{equation}
then the $n \times n$ matrix
\[
[ f_1\bigl( [ x(i), x(j) ] \bigr) ] = %
D \bone_n ( D \bone_n )^T - D H_1\bigl[ \langle x( i )'', x( j )'' \rangle \bigr] D,
\]
where
\[
D := \diag( x( 1 )_0^a, \ldots, x( n )_0^a ), \qquad %
x( i )'' := x( i )'  / x( i )_0 \in \ell^2( \N )
\]
and $H_1$ is an absolutely monotone function on $( -1, 1 )$.

Similarly, if $b \in ( 0, \infty )$ and
\begin{equation}
f_2 : [ 1, \infty ) \to [ 1, \infty ); \ t \mapsto 1 + b \log t
\end{equation}
then the matrix
\begin{align*}
[ f_2\bigl( [ x(i), x(j) ] \bigr) ] & = \bone_{n \times n} + %
\bv \bone_{n \times 1}^T +  \bone_{n \times 1} \bv^T  - %
H_2[ \langle x( i )'', x( j )'' \rangle ] \\[1ex]
 & = %
( \bone_n + \bv ) ( \bone_n + \bv )^T - %
\bv \bv^T - H_2\bigl[ \langle x( i )'', x( j )'' \rangle \bigr],
\end{align*}
where $\bv := ( b \log x( 1 )_0, \ldots, b \log x( n )_0 )^T$,
the vectors $x( 1 )''$, \ldots, $x( n )''$ are as before and
$H_2$ is an absolutely monotone function on $( -1, 1 )$.

Since $\bv \bv^T$ is positive semidefinite,
the Gram matrix $[ \langle x( i )'', x( j )'' \rangle ]$
is positive semidefinite and absolutely monotone functions
preserve positive semidefiniteness,
we see that the functions $f_1$ and $f_2$ are distance preservers.

For a third example, given $p > 0$ and $a > 0$, we let
\begin{equation}
f_3 : [ 1, \infty ) \to [ 1, p ) \subseteq [ 1, \infty ); \ t \mapsto 1 + p - p t^{-a}
\end{equation}
We have that
\begin{align*}
f_3\bigl( [ x, y ] \bigr) & = 1 + p - p ( x_0 y_0 - \langle x', y' \rangle )^{-a} \\[1ex]
 & = 1 + p - p x_0^{-a} y_0^{-a} \Bigl( %
 1 + a \langle x'', y'' \rangle + \frac{a ( a + 1 )}{2} \langle x'', y'' \rangle^2 + \cdots \Bigr),
\end{align*}
where $x = x_0 ( 1, x'' )$ and $y = y_0 ( 1, y'' )$. Therefore, with the previous notation,
\[
\bigl[ f_3( [ x( i ), x( j ) ] ) \bigr] = ( 1 + p ) \bone_n \bone_n^T - %
p D^{-1} H_3\bigl[ \langle x( i )'', x( j )'' \rangle \bigr] D^{-1},
\]
where the absolutely monotone function $H_3$ is such that
\[
H_3( x ) := \sum_{n = 0}^\infty \binom{a + n - 1}{n} x^n \qquad ( -1 < x < 1 ).
\]
Hence $f_3$ is a distance preserver. Moreover, letting $a \to \infty$
gives the existence of a family of discontinuous preservers of the form
\[
t \mapsto \left\{ \begin{array}{ll}
 1 & \text{if } t = 1, \\[1ex]
 1 + p & \text{if } t > 1.
\end{array}\right.
\]

Finally, if $q \in [ 1, \infty )$ then
\[
f_4 : [ 1, \infty ) \to [ 1, \infty ); \ t \mapsto 1 - q + q t
\]
is a distance preserver, since $f_4( 1 ) = 1$, $f_4'( t ) = q > 0$ and
Theorem~\ref{Krein-Gram} shows that
\[
f_4[ M ] = ( 1 - q ) \bone \bone^T + q M = %
q M - ( q - 1 ) \bone \bone^T \in \cM_{1+} \qquad ( M \in \cM_{1+} ).
\]
We can also show directly that $f_4$ preserves $\cM_{1+}$.
Suppose  $M = \bigl[ [ x( i ), x( j ) ] \bigr]_{i, j =1}^n$ for some
$x( 1 )$, \ldots, $x( n ) \in \cL$ and, for each $i \in \{ 1, \ldots, n \}$, let
\[
y( i ) := \bigl( \sqrt{q} x( i )_0, y( i )' \bigr), \quad \text{where }
y( i )' := ( \sqrt{q - 1}, \sqrt{q} x( i )' ) \in \ell^2( \N)
\]
is produced by applying a scaled forward shift to $x( i )'$
and inserting $\sqrt{q - 1}$ as the first coordinate. Then
\[
[ y( i ), y( j ) ] = 1 - q + q [ x( i ), x( j ) ] = f_4[ M ]_{i j} \qquad ( i, j \in \{ 1, \ldots, n \} ).
\]

\subsection{Positive semidefiniteness and conditionally negative definiteness}
If
\[
r : ( 0, \infty ) \to ( 0, \infty ); \ x \mapsto 1 / x
\]
is the reciprocal map and $x = x_0 ( 1, x'' )$, $y = y_0 ( 1, y'' ) \in \cL$ then
\begin{align}
r\bigl( [ x, y ] \bigr) & = x_0^{-1} y_0^{-1} ( 1 - \langle x'', y'' \rangle ) ^{-1} \label{kernel:NP} \\[1ex]
 & = x_0^{-1} y_0^{-1} %
 ( 1 + \langle x'', y'' \rangle + \langle x'', y''\rangle^2 + \cdots ) = %
 x_0^{-1} y_0^{-1} R\bigl( \langle x'', y'' \rangle \bigr) \nonumber
\end{align}
and the function $R$ is absolutely monotone on $( -1, 1 )$.
Hence the entrywise map $r[ - ]$ sends $\cM_{1+}$ into
the cone of positive semidefinite matrices $\cS_+$; for convenience, we
let
\begin{align}
\cS & := %
\bigcup_{n = 1}^\infty\{ A = [ a_{i j} ] \in \R^{n \times n} : a_{i j} = a_{j i} %
\textrm{ for any } i, j \} \\[1ex]
\textrm{and} \quad \cS_+ & := \{ A \in \cS : A \textrm{ is positive semidefinite} \}.
\end{align}
This is a variation on the following result of Bapat \cite[Proof of Lemma~6]{Bapat},
which is a consequence of a result of Micchelli \cite[Corollary~3.1]{Mic}.

\begin{theorem}\label{thm:BM}
If $A = [ a_{i j} ] \in \cS \cap \bigcup_{n = 1}^\infty ( 0, \infty )^{n \times n}$
has exactly one positive eigenvalue then its \emph{Hadamard inverse}
$A^{\circ -1} := [ a_{i j}^{-1} ]$ is positive semidefinite and
\emph{infinitely divisible} \cite{Horn}.
\end{theorem}

\begin{remark}
The Hadamard inverse is also relevant to the theory of complete Nevanlinna--Pick kernels.

Given $x= x_0 ( 1, x'' )$ and $y = y_0 ( 1, y'' ) \in \cL$, we have from (\ref{kernel:NP}) that
\[
\frac{1}{[x,y]} = \frac{x_0^{-1}y_0^{-1}}{1 - \langle x'', y'' \rangle}.
\]
Thus, letting $B$ denote the open unit ball of $\ell^2( \N )$, the Drury--Arveson kernel
\[
B \times B \to ( 0, \infty ); \ ( z, w ) \mapsto \frac{1}{1-\langle z, w \rangle}
\]
can be seen as a perturbation of the restriction of the reciprocal Lorentz form
to the natural embedding $\{ x_0 ( 1, x'' ) \in \cL : x'' \in B \}$ of $B$ into $\cL$.
In particular, the reciprocal Lorentz form
is a complete Nevanlinna--Pick kernel \cite[Theorem~3.10]{AglerMcCarthy}.
\end{remark}

The working above for $f_2$ shows that if $M = [ x( i ), x( j ) ] \in \cM_{1+}$ then
\[
\log M = \log[ \bx_0 ] \bone^T + \bone \log[ \bx_0 ]^T - %
H_2[ \langle x( i )'', x( j )'' \rangle ] \in \cN,
\]
where $\bx_0 = ( x( 1 )_0, \ldots, x( n )_0 )^T$ and
\begin{equation}
\cN := \bigcup_{n = 1}^\infty \{ A \in \R^{n \times n } : %
A \in \cS \textrm{ and } \bv^T A \bv \leq 0 %
\textrm{ for any } \bv \in \C \bone_{n \times 1}^\perp \},
\end{equation}
the cone of conditionally negative definite matrices.

We note that the inclusion
\begin{equation}
\log[ \cM_{1+} ] \subseteq \cN_0 := %
\bigcup_{n = 1}^\infty \{ A = [ a_{i j} ] \in [ 0, \infty )^{n \times n} :  A \in \cN %
\text{ and } a_{i i } = 0 \text{ for each $i$} \}
\end{equation}
is strict: if $B$ is the squared-distance matrix given by the points
$-\sqrt{a}$, $0$ and $\sqrt{a} \in \R$ for some $a > 0$, so that
\[
B = \begin{bmatrix} 0 & a & 4 a \\ a & 0 & a \\ 4 a & a & 0 \end{bmatrix} \in \cN_0,
\]
then
\[
\exp[ B ] = \begin{bmatrix}
 1 & e^a & e^{4 a} \\[1ex]
 e^a & 1 & e^a \\[1ex]
 e^{4 a} & e^a & 1
\end{bmatrix} \not\in \cM_{1+},
\]
by Proposition~\ref{examplematrix}, since
\[
e^{4 a} - 2 e^{2 a} + 1 = ( e^{2 a } - 1 )^2 > 0.
\]
The collection $\cN_0$ is precisely the set of all Euclidean squared-distance matrices;
this is essentially due to Schoenberg \cite{Sch35}.

\begin{remark}
The concept of conditional negative definiteness is fundamental to
the construction of Gaussian fields with stationary increments.
Following Faraut and Harzallah (see \cite{Far}),
Istas~\cite[Theorem~4.1]{Istas} showed that, on the
finite-dimensional real hyperbolic
space~$\cL_n := \bigl( \ell^2( \{ 0, \ldots, n - 1\} ) \times \{ \bzero \} \bigr) \cap \cL$,
the kernel
\[
d^{2 H} : \cL_n \times \cL_n \to \R_+; \ ( x, y ) \mapsto \bigl( \arcosh [ x, y ] \bigr)^{2 H}
\]
is conditionally negative definite, and so hyperbolic fractional Brownian motion with Hurst index $H$ exists,
if and only if $H \in ( 0, 1 / 2 ]$.
\end{remark}

We conclude this section by recording the facts that if $M \in \cM \cap \cN$ then $M \in \cM_{1+}$
and if $M \in \cM^* \cap \cN$ then $M \in \cM^*_{1+}$,
by \cite[Corollary~4.1.5]{BR}.

\subsection{Additive representation for preservers}

The result of Bapat and Micchelli recorded above as Theorem~\ref{thm:BM}
allows us to prove the following
theorem, which is our first step towards finding an additive representation
for preservers.

\begin{theorem}\label{babyadditive}
Let $f : [ 1, \infty ) \to [ 1, \infty )$ have the integral representation
\begin{equation}\label{additive}
f( t ) = C t^\alpha - \int_{\R_+} t^{-s} \std\nu( s ) \qquad ( t \ge 1 ),
\end{equation}
where $\alpha \in [ 0, 1 ]$, the finite Radon measure $\nu$
is supported on $\R_+$ and $C :=1 + \nu( \R_+ )$.
Then $f$ is a distance preserver, with order $\alpha$ and type
$C - \nu\bigl( \{ 0 \}\bigr) 1_{\{ 0 \}}( \alpha )$,
as long as~$\alpha$ and $\nu|_{( 0, \infty )}$ are not both zero.
\end{theorem}
\begin{proof}
It is immediate that $f( 1 ) = 1$ and, as $| s t^{-s -1} | \le ( e \, t_0 \log t_0 )^{-1}$
for all $t \ge t_0 > 1$, we may differentiate under the integral sign to see that
\[
f'( t ) = \alpha C t^{\alpha - 1} + \int_{\R_+} s t^{-s - 1} \std\nu( s ) \qquad ( t > 1 ),
\]
which is positive as long as $\alpha > 0$ or $\nu \neq 0$ on $( 0, \infty )$. Hence $f$
is strictly increasing.

Given any $M \in \cM_{1+}$, we have that
\[
f[ M ] = C M^{\circ \alpha} - \int_{\R_+} M^{\circ -s} \std\nu( s ),
\]
where $[ m_{i j} ]^{\circ x} = [ m^x_{i j } ]$ is the Hadamard power for any $x$.
The first term $C M^{\circ \alpha}$ is of the form $\bv \bv^T - A$,
where $A$ is positive semidefinite, by Theorem~\ref{Krein-Gram}
and the fact that~$f_1$ is a distance preserver. Moreover, Theorem~\ref{thm:BM}
gives that
$M^{\circ -s} = ( M^{\circ -1} )^{\circ s}$ is positive semidefinite for all
$s \ge 0$ and the result follows from another application of
Theorem~\ref{Krein-Gram}; the statement about characteristics is immediate.
\end{proof}

We note the following examples of the additive representation (\ref{additive}),
where $\delta_x$ denotes the measure with unit mass supported at $x$.
\[
\begin{array}{lll}
f_1( t ) \equiv t^a & a \in ( 0, 1 ] & \alpha = a, \ \nu = 0 \\[1ex]
f_3( t ) \equiv 1 + p - p t^{-a} & p, a > 0  & \alpha = 0, \ \nu = p \delta_a \\[1ex]
f_4( t ) \equiv 1 - q + q t & q \ge 1 & \alpha = 1, \ \nu = ( q - 1 ) \delta_0
\end{array}
\]
If $f$ has the representation (\ref{additive}) then $f( t ) = O( t^\alpha )$
as $t \to \infty$, so $f_2$ cannot have such a representation.
However, for any $b > 0$, if $a \in ( 0, 1 )$ is sufficiently small then taking
the measure $\nu = ( a^{-1} b - 1 ) \delta_0$ in (\ref{additive}) gives
\[
f( t ) = a^{-1} b t^a - a^{-1} b + 1 = %
1 + b \frac{e^{a \log t} - 1}{a} \to 1 + b \log t = f_2( t ) \qquad \text{as } a \to {0+}.
\]
This motivates the introduction of the following two families of preservers.

Given any $\alpha \in ( 0, 1 ]$,
any finite positive Radon measure $\nu$ supported on $( 0, \infty )$ and
any $\beta \geq \alpha ( 1 + \nu\bigl( ( 0, \infty ) \bigr) )$, setting
\[
f_{\alpha, \beta, \nu}( t ) := %
1 + \beta \frac{t^\alpha - 1}{\alpha} + \int_{( 0, \infty )} ( 1 - t^{-s} ) \std\nu( s ) %
\qquad ( t \geq 1 )
\]
defines a preserver: note that
\[
f_{\alpha, \beta, \nu}'( t ) = %
\beta t^{\alpha - 1} + \int_{( 0, \infty )} s t^{-s - 1} \std\nu( s ) \qquad ( t > 1 )
\]
and
\[
f_{\alpha, \beta, \nu}[ M ] = %
\frac{\beta}{\alpha} M^{\circ \alpha} - %
\Bigl( \frac{\beta}{\alpha} - 1 - \nu\bigl( ( 0, \infty ) \bigr) \Bigr) \bone \bone^T - %
\int_{( 0, \infty )} M^{\circ -s} \std\nu( s ).
\]

For $\alpha = 0$ we require $\beta \ge 0$ and $\nu$ need not be finite
but it satisfies the following integrability condition:
\[
\int_{( 0, \infty )} \min\{ s, 1 \} \std\nu( s ) < \infty \quad \iff \quad
\int_{( 0, \infty )} \frac{s}{1 + s} \std\nu( s ) < \infty.
\]
We also insist that $\beta$ and $\nu$ cannot both be zero. Then setting
\[
f_{0, \beta, \nu}( t ) := %
1 + \beta \log t + \int_{( 0, \infty )} ( 1 - t^{-s} ) \std\nu( s ) %
\qquad ( t \ge 1 )
\]
gives a preserver. To see this, we first write
$f( t ) := f_{0, \beta, \nu}( t ) = 1 + B( \log t )$,
where
\[
B( x ) := \beta x + \int_{( 0, \infty )} ( 1 - e^{-x s} ) \std\nu( s ) \qquad ( x \ge 0 ).
\]
Thus $B$ is a Bernstein function with $B( 0 ) = 0$ (see \cite[Chapter~3]{SSV})
and
\[
f[ M ] = \bone \bone^T + B\bigl[ \log[ M ] \bigr] %
\qquad \text{for any } M \in \cM_{1+}.
\]
As noted above, we have that $\log[ M ] \in \cN_0$,
the closed cone of conditionally negative definite matrices with non-negative entries
and zeros on the leading diagonal. We know that Bernstein functions
are entrywise preservers of this class, because $\exp[ -s A ]$
is positive semidefinite for any $A \in \cN_0$ and any $s > 0$,
and so $f[ M ] \in \cN \cap \cM \subseteq \cM_{1+}$.

\subsection{Closure under composition}
The collection of distance preservers is closed under composition.
We show now that this is the case for the collection of preservers
that we have just introduced.

If $g_1( t ) = 1 + B_1( \log t )$ and $g_2( t ) = 1 + B_2( \log t )$, where
$B_1$ and $B_2$ are Bernstein functions that vanish at the origin, then
\[
g_2\bigl( g_1( t ) \bigr) = 1 + B_2( \log\bigl( 1 + B_1( \log t ) \bigr) ).
\]
Since $x \mapsto \log( 1 + x )$ is a Bernstein function that vanishes at the
origin and the class of Bernstein functions is closed under composition,
the function $g_2 \comp g_1$ is of the same form.

Next we suppose $\alpha \in ( 0, 1 ]$, $\beta \ge 0$ and $\nu$
is a finite positive Radon measure on $\R_+$
and write
\[
f( t ) = %
C t^\alpha - \int_{\R_+} t^{-s} \std\nu( s ) = %
C t^\alpha \Bigl( 1 - \frac{1}{C} \int_{\R_+} t^{-s - \alpha} \std\nu_0( s ) \Bigr),
\]
where
\[
C \ge a := 1 + \nu\bigl( ( 0, \infty ) \bigr) \quad \text{and} \quad %
\nu_0 : A \mapsto ( C - a ) \delta_0( A ) + \nu\bigl( A \cap ( 0, \infty ) \bigr).
\]
We also let
\[
\psi( t ) := \frac{1}{C} \int_{\R_+} t^{-s - \alpha} \std\nu_0( s ) = %
\int_{\R_+} t^{-r} \std\tilde{\nu}_\alpha( r ),
\]
where $\tilde{\nu}_\alpha : A \mapsto C^{-1} \nu_0\bigl(  ( A - \alpha ) \cap \R_+ \bigr)$,
and we have that
\[
0 \leq \psi( t ) \leq \psi( 1 ) = \frac{C - 1}{C} < 1 \qquad ( t \ge 1 ).
\]
If $g$ is another function of the same form, so that
\[
g( t ) = C' t^{\alpha'} - \int_{\R_+} t^{-s} \std\nu_0'( s ) \qquad ( t \ge 1 ),
\]
then
\begin{equation}\label{gcompf}
g\bigl( f( t ) \bigr) = C' f( t )^{\alpha'} - \int_{\R_+} f( t )^{-s} \std\nu_0'( s ).
\end{equation}
We know that, for $| z | < 1$,
\[
1 - ( 1 - z )^{\alpha'}= %
\sum_{n = 1}^\infty c_n z^n, \quad \text{where } %
c_n:= \frac{\alpha' ( 1 - \alpha' ) \cdots ( n - 1 - \alpha' )}{n!} \ge 0.
\]
Hence we can write
\[
C^{\alpha'} t^{\alpha \alpha'} - f( t )^{\alpha'} = %
C^{\alpha'} t^{\alpha \alpha'} ( 1 - \bigl( 1 - \psi( t ) \bigr)^{\alpha'}  ) = %
C^{\alpha'} \sum_{n = 1}^\infty c_n t^{\alpha \alpha'} \psi( t )^n,
\]
and therefore
\[
f( t )^{\alpha'} = C^{\alpha'} t^{\alpha \alpha'} - \int_{\R_+} t^{-s} \std\sigma( s ),
\]
where, with $\star$ denoting convolution on $\R_+$, the measure
\[
\sigma : A \mapsto %
C^{\alpha'} \sum_{n = 1}^\infty c_n \tilde{\nu}_\alpha^{\star n}( A + \alpha \alpha' )
\]
has total mass $C^{\alpha'} - 1$ (which may be seen by taking $t = 1$
in the previous expression for $f( t )^{\alpha'}$).

For the second term on the right-hand side of (\ref{gcompf}),
if $s \ge 0$ then
\[
f( t )^{-s} = C^{-s} t^{-s \alpha} \bigl( 1 - \psi( t ) \bigr)^{-s} = %
C^{-s} t^{-s \alpha} \sum_{n = 0}^\infty d_n( s ) \psi( t )^n,
\]
where
\[
d_n( s ) = ( -1 )^n \binom{-s}{n} = \frac{s ( 1 + s ) \cdots ( n - 1 + s )}{n!} \ge 0.
\]
Thus
\[
f( t )^{-s} = \int_{\R_+} t^{-r} \std\tau_s( r ), \quad \text{where } %
\tau_s : A \mapsto C^{-s} \sum_{n = 0}^\infty %
d_n( s ) \tilde{\nu}_\alpha^{\star n}\bigl( ( A - s \alpha ) \cap \R_+ \bigr),
\]
and taking $t = 1$ shows that $\tau_s$ is a probability measure. It follows that
\begin{align*}
g\bigl( f( t ) \bigr) & = C' C^{\alpha'} t^{\alpha \alpha'} - %
C' \int_{\R_+} t^{-s} \std\sigma( s ) - %
\int_{\R_+} \int_{\R_+} t^{-r} \std\tau_s( r ) \std\nu_0'( s ) \\[1ex]
 & = C' C^{\alpha'} t^{\alpha \alpha'} - \int_{\R_+} t^{-s} \std\rho( s ),
\end{align*}
by Tonelli's theorem, where
\[
\rho := C' \sigma + \int_{\R_+} \tau_s \std\nu_0'( s )
\]
has total mass $C' ( C^{\alpha'} - 1 ) + ( C' - 1 ) = C' C^{\alpha'} - 1$, as required.

We next consider the mixed cases. If
\[
f( t ) = 1 + C ( t^\alpha - 1 ) + \int_{( 0, \infty )} ( 1 - t^{-s} ) \std\nu( s ) %
\qquad ( t \ge 1 )
\]
and $g( t ) \equiv 1 + B( \log t )$ for a Bernstein function $B$ that vanishes at the
origin then $f \comp g$ has the same form, since
\[
x \mapsto ( 1 + x )^\alpha - 1 \qquad \text{and} \qquad %
x \mapsto 1 - ( 1 + x )^{-s}
\]
are Bernstein functions that vanish at the origin and this class forms a cone
that is closed under pointwise limits. Finally, as
\[
f( t ) = C t^\alpha \bigl( 1 - \psi( t ) \bigr) \quad \implies \quad %
F( x ) := \log f( e^x ) = \log C + \alpha x + \log\bigl( 1 - \psi( e^x ) \bigr),
\]
it suffices to show that $F'$ is completely monotone \cite[Chapter~1]{SSV}.
We can write $\xi( x ) := \psi( e^x )$ and then
\[
F'( x ) = \alpha + \frac{\rd}{\rd x} \log\bigl( 1 - \psi( e^x ) \bigr) = %
\alpha + \frac{-\xi'( x )}{1 - \xi( x )} = %
\alpha - \xi'( x ) \sum_{n = 0}^\infty \xi( x )^n,
\]
where
\[
\xi( x ) = \int_{\R_+} e^{-r x} \std\tilde{\nu}_\alpha( r ) \in [ 0, 1 ) %
\qquad ( x > 0 ).
\]
It is immediate that $\xi$ is completely monotone
and therefore so are the functions $-\xi'$ and $\xi^n$ for any $n \ge 1$,
as the class of completely monotone functions is closed under products.
As it is also closed under sums and pointwise convergence, we have the
closure result claimed.

\begin{notation}
We let $\cR_1$ denote the collection of finite positive Radon measures
on~$\R_+$ and $\cR_2$ the collection of positive Radon measures
on~$( 0, \infty )$ that satisfy the L\'evy--Khintchine integrability condition:
such a measure $\nu$ belongs to $\cR_2$ if and only if
\[
\int_{( 0, \infty )} \min\{ s, 1 \} \std\nu( s ) < \infty.
\]
The working above shows that $\cF := \cF_1 \cup \cF_2$ is closed
under composition, where
\begin{align}
\cF_1 & := \Bigl\{ f : t \mapsto %
C t^\alpha - \int_{\R_+} t^{-s} \std\nu_1( s ) \mid %
\alpha \in ( 0, 1 ], \ \nu_1 \in \cR_1, \ %
C = 1 + \nu_1( \R_+ ) \Bigr\} \\[1ex]
\text{and} & \nonumber \\[1ex]%
\cF_2 & := \Bigl\{ f : t \mapsto %
1 + \beta \log t + \int_{( 0, \infty )} ( 1 - t^{-s} ) \std\nu_2( s ) \mid %
( \beta, \nu_2 ) \in ( \R_+ \times \cR_2 ) \setminus \{ ( 0, 0 ) \} \Bigr\}.
\end{align}
If $\cB_0$ denotes the set of non-zero Bernstein functions that vanish at $0$,
\begin{equation}
\Bigl\{ f : \R_+ \to \R_+; \ %
x \mapsto \beta x + \int_{( 0, \infty )} ( 1 - e^{-x s} ) \std\nu( s ) \mid %
( \beta, \nu ) \in ( \R_+ \times \cR_2 ) \setminus \{ ( 0, 0 ) \} \Bigr\},
\end{equation}
then we have that
\[
\cF_2 = \{ f : t \mapsto 1 + B( \log t ) \mid B \in \cB_0 \}.
\]
\end{notation}

We can unify the representations of $\cF_1$ and $\cF_2$ if we make the
following definition:
\[
\frac{t^\alpha - 1}{\alpha}\Bigr|_{\alpha = 0} := %
\lim_{\alpha \to 0+} \frac{t^\alpha - 1}{\alpha} = \log t \qquad ( t \ge 1).
\]
Then $f \in \cF$ if and only if
\begin{equation}\label{uniform}
f( t ) = %
1 + \beta \frac{t^\alpha - 1}{\alpha} + \int_{( 0, \infty )} ( 1 - t^{-s} ) \std\nu( s ),
\end{equation}
where $\alpha \in [ 0, 1 ]$,
$( \beta, \nu ) \in ( \R_+ \times \cR_2 ) \setminus \{ ( 0, 0 ) \}$
and $\beta \ge \alpha ( 1 + \nu\bigl( ( 0, \infty ) \bigr) )$,
and this inequality is always satisfied if $\alpha = 0$.

This form is immediate if $\alpha = 0$, and if $\alpha \in ( 0, 1 ]$ then $\nu$
must be finite, by the inequality, and we can write $f$ as an element of $\cF_1$ with
\[
C = \frac{\beta}{\alpha} \quad \text{and} \quad %
\nu_1 : A \mapsto ( C - 1 - \nu\bigl( ( 0, \infty ) \bigr) ) \delta_0( A ) + %
\nu\bigl( A \cap ( 0, \infty ) \bigr).
\]

We call $\alpha$ the \emph{order}, $\beta$ the \emph{drift} and $\nu$
the \emph{L\'evy measure} of $f$. We show below in Section~\ref{sec:uniqueness}
that this terminology is appropriate.

\subsection{Closure under pointwise limits}

Since $\cM_{1+}$ is closed under pointwise limits, 
so is its class of entrywise preservers. In this section we determine
which functions can be obtained as pointwise limits of sequences of elements
of the class $\cF$. We obtain a collection of distance preservers which
includes functions with a jump discontinuity at~$1$, which is generated by mass
from the L\'evy measures escaping to infinity.

Suppose $f_n$ has the form (\ref{uniform}) for any $n \in \N$, with
parameters $( \alpha_n, \beta_n, \nu_n )$, and suppose that $\bigl( f_n( t ) \bigr)$
is convergent, with limit $f( t )$, for all $t \ge 1$.

It is immediate that $f( 1 ) = 1$. Since $( \alpha_n )$ is bounded, we may pass to a subsequence and assume that $\alpha_n \to \alpha \in [ 0, 1 ]$. Furthermore, we have that
\[
f_n( 2 ) \ge 1 + \beta_n \frac{2^{\alpha_n} - 1}{\alpha_n} \ge 1 + \beta_n \log 2 \quad \implies \quad %
\frac{f_n( 2 ) - 1}{\log 2} \geq \beta_n \geq 0,
\]
so we may pass to a further subsequence and assume that
$\beta_n \to \beta \in \R_+$.

For the sequence of measures $( \nu_n )$, note first
that setting
\[
\pi_n : A \mapsto \int_{A \cap ( 0, \infty )} ( 1 - e^{-s} ) \std\nu_n( s )
\]
defines a finite measure on $[ 0, \infty ]$, since
\[
\frac{s}{e} \le 1 - e^{-s} \le s \quad \bigl( s \in [ 0, 1 ] \bigr) %
\quad \text{and} \quad %
1 - e^{-1} \le 1 - e^{-s} \le 1 \quad \bigl( s \in [ 1, \infty ) \bigr).
\]
[This also shows that $\nu_n \in \cR_2$ if and only if $\pi_n$ is finite.]
Furthermore, we have that
\[
\pi_n\bigl( [ 0, \infty ] \bigr) = f_n( e ) - 1 - \beta_n \frac{e^{\alpha_n} - 1}{\alpha_n}
\]
and so the sequence $( \pi_n )$ is bounded in $C[ 0, \infty ]^*$. Passing
to a subsequence, there exists a finite positive measure $\pi$ on $[ 0, \infty ]$
such that $\int_{[ 0, \infty ]} g \std\pi_n \to \int_{[ 0, \infty]} g \std \pi$ for any
continuous function $g : [ 0, \infty ] \to \R$.

In particular, for any $t \in ( 1, \infty )$ we may take
\[
\chi_0 : [ 0, \infty ] \to \R; \ s \mapsto \left\{\begin{array}{cl}
 \log t & \text{if } s = 0, \\[1ex]
\displaystyle \frac{1 - t^{-s}}{1 - e^{-s}} & \text{if } 0 < s < \infty, \\[2ex]
 1 & \text{if } s = \infty.
\end{array}\right.
\]
Then
\begin{align*}
f_n( t ) - 1 - \beta_n \frac{t^{\alpha_n} - 1}{\alpha_n} & = %
\int_{( 0, \infty )} ( 1 - t^{-s} ) \std\nu_n( s ) \\[1ex]
 & = \int_{[ 0, \infty ]} \chi_0( s ) \std\pi_n( s ) \\[1ex]
 & \to \int_{[ 0, \infty ]} \chi_0( s ) \std\pi( s )
\end{align*}
and therefore
\[
f( t ) = 1 + \beta \frac{t^\alpha - 1}{\alpha} + %
\pi\bigl( \{ 0 \} \bigr) \log t + \int_{( 0, \infty )} ( 1 - t^{-s} ) \std\nu( s ) + %
\pi\bigl( \{ \infty \} \bigr) \qquad ( t > 1 ),
\]
where $\nu \in \cR_2$ is such that
$\std\nu( s ) = ( 1 - e^{-s} )^{-1} \std\pi( s )$.

We now claim that if $\alpha > 0$ then $\pi$ has no mass at the origin
and $\nu$ is finite, with
\[
\beta \ge \alpha ( 1 + \nu\bigl( ( 0, \infty ) \bigr) + \pi\bigl( \{ \infty \} \bigr) ).
\]
To see this, we may first assume that $\alpha_n > 0$ for all $n$ and then
\[
\nu_n\bigl( ( 0, \infty ) \bigr) \leq \frac{\beta_n}{\alpha_n} - 1 %
\qquad \text{for any } n \in \N,
\]
so each measure $\nu_n$ is finite
and their total masses are bounded above by $M$, say. Now
suppose $\delta > 0$, $\eps > 0$ and the continuous function $\chi_\eps$
is such that
\[
\chi_\eps : [ 0, \infty ] \to [ 0, 1 ]; \ s \mapsto \left\{\begin{array}{cl}
 1 & \text{if } 0 \le s \le \eps, \\[1ex]
 0 & \text{if } s \ge 2 \eps.
\end{array}\right.
\]
Then
\[
\pi\bigl( \{ 0 \} \bigr) \le %
\int_{[ 0, \infty ]} \chi_\eps( s ) \std\pi( s ) = \lim_{n \to \infty} %
\int_{[ 0, \infty ]} \chi_\eps( s ) \std\pi_n( s ) \leq %
\int_{[ 0, \infty ]} \chi_\eps( s ) \std\pi_m( s ) + \delta
\]
for some $m \in \N$ and
\[
\int_{[ 0, \infty]} \chi_\eps( s ) \std\pi_m( s ) = %
\int_{( 0, \infty )} \chi_\eps( s ) ( 1 - e^{-s} ) \std\nu_m( s ) \le ( 1 - e^{-2 \eps} ) M.
\]
As this upper bound may be made arbitrarily small by shrinking $\eps$,
we see that $\pi$ has no mass at the origin, as claimed.

For the inequality, let $\delta > 0$ and $\eps > 0$
and consider the continuous function
\[
\chi_\infty : [ 0, \infty ] \to \R_+; \ s \mapsto \left\{\begin{array}{cl}
 ( 1 - e^{-\eps} )^{-1} s / \eps & \text{if } 0 \le s \le \eps, \\[1ex]
 ( 1 - e^{-s} )^{-1} & \text{if } \eps \le s < \infty, \\[1ex]
 1 & \text{if } s = \infty.
\end{array}\right.
\]
Then
\[
\nu\bigl( [ \eps, \infty ) \bigr) + \pi\bigl( \{ \infty \} \bigr) %
\le \int_{[ 0, \infty ]} \chi_\infty( s ) \std\pi( s ) %
\le \int_{[ 0, \infty ]} \chi_\infty( s ) \std\pi_m( s ) + \delta
\]
for some sufficiently large $m$. Now,
\begin{align*}
\int_{[ 0, \infty ]} \chi_\infty( s ) \std\pi_m( s ) & = %
\int_{( 0, \infty )} \chi_\infty( s ) ( 1 - e^{-s} ) \std\nu_m( s ) \\[1ex]
 & = \int_{( 0, \eps )} \frac{( 1 - e^{-s} ) s}{( 1 - e^{-\eps} ) \eps} \std\nu_m( s ) %
+ \nu_m\bigl( [ \eps, \infty ) ) \le \nu_m\bigl( ( 0, \infty ) \bigr) \le M.
\end{align*}
This shows that $\nu$ is finite and also gives the bound claimed.

It follows that if $f : [ 1, \infty ) \to [ 1, \infty )$ lies in
the pointwise closure $\overline{\cF}$ of $\cF$ then $f( 1 ) = 1$ and
\begin{equation}\label{preserver}
f( t ) = 1 + \gamma + \beta \frac{t^\alpha - 1}{\alpha} + %
\int_{ ( 0, \infty )} ( 1 - t^{-s} ) \std\nu( s ) \qquad ( t > 1 ),
\end{equation}
where $\alpha \in [ 0, 1 ]$, $( \beta, \nu ) \in \R_+ \times \cR_2$
and $\gamma \in \R_+$,
with $\beta \ge \alpha ( 1 + \gamma + \nu\bigl( ( 0, \infty ) \bigr) )$
and this inequality is always satisfied if $\alpha = 0$.

If $\gamma > 0$ then this function is discontinuous at $1$,
but this is the only discontinuity that can arise. To see that every such $f$
can appear, we note that if
$\beta \ge \alpha ( 1 + \gamma + \nu\bigl( ( 0, \infty ) \bigr) )$ and
$f_n \in \cF$ has parameters
$( \alpha, \beta, \nu_n := \nu + \gamma \delta_n )$ then
\[
\beta \ge \alpha ( 1 + \gamma + \nu\bigl( ( 0, \infty ) \bigr)  ) = %
\alpha ( 1 + \nu_n\bigl( ( 0, \infty ) \bigr) ) 
\]
and if $t > 1$ then
\begin{align*}
f_n( t ) & = 1 + \beta \frac{t^\alpha - 1}{\alpha} + %
\int_{( 0, \infty )} ( 1 - t^{-s} ) \std\nu_n( s ) \\[1ex]
 & = 1 + \beta \frac{t^\alpha - 1}{\alpha} + %
 \int_{( 0, \infty )} ( 1 - t^{-s} ) \std\nu( s ) + \gamma ( 1 - t^{-n} ) %
\to f( t ) \quad \text{as } n \to \infty.
\end{align*}

We also have a direct proof that every $f$ of the form (\ref{preserver})
preserves $\cM_{1+}$. Given any $M = [ m_{i j} ]_{i, j = 1}^n \in \cM_{1+}$,
we let $N = [ n_{i j} ]_{i, j =1}^n \in \cS$ be defined by setting
\[
n_{i j} := \left\{\begin{array}{ll}
 1 & \text{if } m_{i j} = 1, \\[1ex] 0 & \text{if } m_{i j} \neq 1.
\end{array}\right.
\]
Since $m_{i i} = 1$ for all $i$, we see that
\[
\begin{vmatrix} 1 & n_{i j} \\[1ex] n_{i j} & 1 \end{vmatrix} = 1 - n_{i j}^2 \ge 0
\]
and
\[
\begin{vmatrix}
 1 & n_{i j} & n_{k i}\\[1ex]
 n_{i j} & 1 & n_{j k} \\[1ex]
 n_{k i} & n_{j k} & 1
\end{vmatrix} = %
1 + 2 n_{i j} n_{j k} n_{k i} - ( n_{i j}^2 + n_{j k}^2 + n_{k i}^2 ) \ge 0;
\]
if none, two or three of $n_{i j}$, $n_{j k}$ and $n_{k i}$ are zero then this is immediate,
and if only one is zero then the corresponding principal submatrix $A$ of $M$
has determinant $-( m_{i j} - 1 )^2$, say, where $m_{i j} > 1$,
contradicting the fact that
$A \in \cM_{1+}$. This shows that $N$ is $3$-PMP
and therefore positive semidefinite \cite[Theorem~2.7]{BGKP19}.

If we define $\tilde{f}$ on $[ 1, \infty )$ by setting $\tilde{f}( 1 ) = 1 + \gamma$
and $\tilde{f}( t ) = f( t )$ for all $t > 1$, then
\[
f[ M ] = \tilde{f}[ M ] - \gamma N,
\]
so it suffices to show that $\tilde{f}[ - ]$ sends $\cM_{1+}$ into $\cM_{1+}^*$.
We now split into two cases.

If $\alpha > 0$ then
\[
\tilde{f}( t ) = C t^\alpha - \int_{\R_+} t^{-s} \std\nu_0( s ) \qquad ( t \ge 1 ), 
\]
where
\[
C = \frac{\beta}{\alpha} \qquad \text{and} \qquad %
\nu_0 : A \mapsto ( C - \nu\bigl( ( 0, \infty ) \bigr) - 1 - \gamma ) \delta_0( A ) %
+ \nu\bigl( A \cap ( 0, \infty ) \bigr).
\]
The argument proceeds exactly as before.

For $\alpha = 0$, again the argument is unchanged: we have that
\[
\tilde{f}[ M ] = ( 1 + \gamma ) \bone \bone^T + B\bigl[ \log [ M ] \bigr]
\]
for some Bernstein function that vanishes at the origin, so $\tilde{f}[ M ] \in \cN$
and again we have that $\tilde{f}[ M ] \in \cM_{1+}^*$, by \cite[Corollary~4.1.5]{BR}.

We note that, for any $f$ satisfying (\ref{preserver}), the inequality
$\beta \ge \alpha ( 1 + \gamma + \nu\bigl( ( 0, \infty ) \bigr)  )$ is equivalent
to requiring the function $t \mapsto t^{-\alpha} f( t )$ to be non-decreasing
on $( 1, \infty )$. If $f$ is of the form (\ref{preserver}) then
\begin{align*}
t & f'( t ) - \alpha f ( t ) \\[1ex]
 & = %
t \Bigl( \beta t^{\alpha - 1} + \int_{( 0, \infty )} s t^{-s - 1} \std\nu( s ) \Bigr) - %
\alpha ( 1 + \gamma ) - \beta ( t^{\alpha} - 1 ) - %
\alpha \int_{( 0, \infty )} ( 1 - t^{-s} ) \std\nu( s ) \\[1ex]
 & = \beta - \alpha ( 1 + \gamma + \nu\bigl( ( 0, \infty ) \bigr)  ) + %
\int_{( 0, \infty )} ( s + \alpha ) t^{-s} \std\nu( s ) \ge 0 \qquad \text{for any } t > 1
\end{align*}
if and only if $\beta \ge \alpha ( 1 + \gamma + \nu\bigl( ( 0, \infty ) \bigr)  )$,
since $\int_{( 0, \infty )} ( s + \alpha ) t^{-s}\ \std\nu( s ) \to 0$ as $t \to \infty$,
by the monotone convergence theorem. To see that
$s \mapsto ( s + \alpha ) t^{-s}$ is $\nu$-integrable for any~$t > 1$,
we note that this function is bounded on $( 0, \infty )$ and has a limit at the origin,
since
\[
\lim_{s \to {0+}} ( s + \alpha ) t^{-s} = \alpha \qquad \text{and} \qquad %
\lim_{s \to \infty} ( s + \alpha ) t^{-s} = 0.
\]
It remains to consider integrability on $( 0, 1 )$ if $\nu$ is infinite,
but then $\alpha = 0$ and
\[
0 \le ( s + \alpha ) t^{-s} \le s \qquad \text{for any } s > 0.
\]

\subsection{Uniqueness of the additive representation}\label{sec:uniqueness}

If the map $f : [ 1, \infty ) \to [ 1, \infty )$ has the representation (\ref{preserver})
and $f( 1 ) = 1$ then
\[
\gamma =f( {1+} ) - 1, \qquad
\alpha = \lim_{t \to \infty} \frac{\log f( t )}{\log t} \qquad \text{and} \qquad %
\beta = \lim_{t \to \infty} t^{1 - \alpha} f'( t ).
\]
Given these parameters, the uniqueness of the measure $\nu$ follows from
standard Laplace-transform arguments.

The first claim is immediate. For the second, we let
$F( x ) := \log f( e^x )$ for all $x > 0$ and note that, if $\alpha = 0$,
\[
0 \le \frac{F( x )}{x} = %
\frac{\log x}{x} + \frac{1}{x} \log\Bigl( x^{-1} f( {1+} ) + \beta + %
\int_{( 0, \infty )} \frac{1 - e^{-x s}}{x} \std\nu( s ) \Bigr)
\]
and, if $\alpha \in ( 0, 1 ]$,
\[
\frac{F( x )}{x} = %
\alpha + \frac{1}{x} \log\Bigl( e^{-\alpha x} f( {1+} ) + %
\beta \alpha^{-1} ( 1 - e^{-\alpha x} ) + %
x e^{-\alpha x} \int_{( 0, \infty )} \frac{1 - e^{-x s}}{x} \std\nu( s ) \Bigr).
\]
For the third, we have that
\[
t^{1 - \alpha} f'( t ) = %
\beta + t^{-\alpha} \int_{( 0, \infty )} s t^{-s} \std\nu( s ) \to \beta %
\qquad \text{as } t \to \infty,
\]
following the same manner of working as at the start of the proof of Theorem~\ref{babyadditive}.

We also note that $f$ is of finite type if $\alpha > 0$, since
\[
C := \lim_{t \to \infty} t^{-\alpha} f( t ) = \frac{\beta}{\alpha}
\]
in this case, and if $\alpha = 0$ then $f$ has finite type
$C = 1 + \gamma + \nu\bigl( ( 0, \infty ) \bigr)$ if and only
if~$\beta = 0$ and $\nu$ is finite, since in this case
\[
f( t ) = %
1 + \gamma + \beta \log t + \int_{( 0, \infty )} ( 1 - t^{-s} ) \std\nu( s ) %
\qquad ( t > 1 ).
\]
 
\section{Three consequences for positive-semidefinite kernels}

Positive-semidefinite kernels may be obtained naturally from distance preservers.
Here, we record three consequences of this observation:
a connection with complete Nevanlinna--Pick kernels,
a perturbation result that produces new preservers from old
and the non-existence of certain fractional Ornstein--Uhlenbeck fields.

\subsection{Complete Nevanlinna--Pick kernels}\label{psdkernels}

Consider a kernel on $\cL$ of the form
\begin{equation}\label{psdkernel}
G : \cL \times \cL \to \R_+; \ ( x, y ) \mapsto g\bigl( d( x, y ) \bigr),
\end{equation}
where $g : \R_+ \to \R_+$ is continuous. We may write
\[
G( x, y ) = g\bigl( d( x, y ) \bigr) = g\bigl( \arcosh [ x, y ] \bigr) = h\bigl( [ x, y ] \bigr),
\]
where $h := g \comp \arcosh : [ 1, \infty ) \to \R_+$.

The kernel $G$ is \emph{positive semidefinite} if and only if the matrix
\[
[ G\bigl( x( i ), x( j ) \bigr) ]_{i, j = 1}^n \in \cS_+
\]
for any finite collection $x( 1 )$, \ldots, $x( n ) \in \cL$.

If $f : [ 1, \infty ) \to [ 1, \infty )$ is a distance preserver 
then its continuous regularisation $\tilde{f}$
maps $\cM_{1+}$ to $\cM_{1+}^*$
and the entrywise reciprocal $r[ - ]$ maps $\cM_{1+}^*$ into~$\cS_+$.
Thus if we set $g := r \comp \tilde{f} \comp \cosh$ then the kernel
\begin{equation}\label{kernel:G}
G : \cL \times \cL \to ( 0, \infty ); \ ( x, y ) \mapsto \frac{1}{\tilde{f}\big( [ x, y ] \bigr)}
\end{equation}
is positive semidefinite. This kernel is also \emph{infinitely divisible}, which may be
seen by noting that $f_1 \comp r \comp \tilde{f} = r \comp f_1 \comp \tilde{f}$, where~$f_1$ is as in Section~\ref{examples}. 

The kernel $G$ is, in fact, a complete Nevanlinna--Pick kernel.
By the Quiggin--McCullough criterion
\cite[Corollary 1.12]{AglerMcCarthy}, a positive-semidefinite
kernel $K$ that takes positive values is a complete Nevanlinna--Pick kernel if and only if, for
every finite collection of distinct points $x(1)$, \ldots, $x(n)$, the matrix
\[
\Bigl[ \frac{1}{K\bigl( x( i ), x( j ) \bigr)} \Bigr]_{i, j = 1}^n
\]
has exactly one positive eigenvalue.
For $G$ this matrix is the image of a Lorentz--Gram matrix under the action of the preserver $\tilde{f}$,
so~$G$ is a complete Nevanlinna--Pick kernel. This reasoning extends to provide the following result.

\begin{proposition}
Let $f : [ 1, \infty ) \to [ 1, \infty )$ be a distance preserver and let
\[
K : X \times X \to ( 0, \infty )
\]
be a complete Nevanlinna--Pick kernel such that $K( x, x ) \le 1$ for every $x \in X$.
Then
\[
K_f : X \times X \to ( 0, \infty ); \ ( x, y ) \to f\bigl( K( x, y )^{-1} \bigr)^{-1}
\]
is also a complete Nevanlinna--Pick kernel.
\end{proposition}
\begin{proof}
We note first that, as $K$ is positive semidefinite,
\[
0 < K( x, y )^2 \le K( x, x ) K( y, y ) \le 1 \qquad ( x, y \in  X )
\]
and so $K_f$ is well defined. Next, given $x(1)$, \ldots, $x(n) \in X$, the Quiggin--McCullough criterion
\cite[Corollary 1.12]{AglerMcCarthy} gives that
\[
M := \bigl[ K( x( i ), x( j ) \bigr)^{-1} \bigr]_{i, j = 1}^n \in \cM_{1+}^*.
\]
Since $f$ is a distance preserver, we have that $f[ M ] \in \cM_{1+}^*$.
The reciprocal map $r[ - ]$ sends $\cM_{1+}^*$ into $\cS_+$, and
therefore
\[
\bigl[ K_f\bigl( x( i ), x( j ) \bigr) \bigr]_{i, j = 1}^n = ( r \comp f )[ M ] \in \cS_+.
\]
Thus $K_f$ is positive semidefinite and its reciprocal matrix is $f[ M ]$,
which has exactly one positive eigenvalue. Another application of the Quiggin--McCullough criterion
gives the result.
\end{proof}

\subsection{Perturbation of distance preservers}

Krein \cite[(4.14)]{Krein-kernels} states that
a kernel $G$ of the form (\ref{psdkernel}) is positive semidefinite if and only if
\begin{equation}\label{psd:rep}
h : [ 1, \infty ) \to \R_+; \ t \mapsto \int_{\R_+} t^{-s} \std \mu( s ) = %
\int_{\R_+} e^{-s \log t} \std \mu( s )
\end{equation}
for some finite positive Radon measure $\mu$ on $\R_+$
and this representation is unique. A proof of the existence part of this claim
was given by Faraut and Harzallah \cite[Th\'eor\`eme~8.1]{FH};
uniqueness follows from the uniqueness of the Laplace transform.

By Bernstein's theorem, this integral-representation theorem shows that
the kernel~$G$ is positive semidefinite if and only if the function
$x \mapsto h( e^x )$ is completely monotone.
Then $x \mapsto h( e^{-x} )$ is absolutely monotone on $( -\infty, 0 )$,
so is real analytic there. It follows that $h$ is real analytic on $( 1, \infty )$
and $g$ is real analytic on $( 0, \infty )$, since the functions
$\log$, $x \mapsto -x$ and $\cosh$ are real analytic also.

In particular, if $f$ is a distance preserver then
$x \mapsto 1 / \tilde{f}( \cosh x ) = 1 / f( \cosh x )$
is real analytic on~$( 0, \infty )$ and so $f$ is real analytic on $( 1, \infty )$.
We exploit this observation in Section~\ref{sec:character}.

\begin{corollary}\label{minuspsd}
If $f : [ 1, \infty ) \to [ 1, \infty )$ is a distance preserver and
$h : [ 1, \infty ) \to \R_+$ has the representation \textup{(\ref{psd:rep})}
for some finite positive Radon measure $\mu$ on $\R_+$
then
\[
k : [ 1, \infty ) \to [ 1, \infty ); \ t \mapsto a f( t ) - h( t )
\]
is also a distance preserver, where $a := 1 + \mu( \R_+ )$.
\end{corollary}
\begin{proof}
Note first that $h( t ) \leq \mu( \R_+ )$ for any $t \geq 1$
and so $k( t ) \geq a - \mu( \R_+ ) = 1$.
We also have that $k( 1 ) = a - h( 1 ) = a - \mu( \R_+ ) = 1$
and if $t > 1$ then $f( t ) > 1$ and $h( t ) \leq h( 1 )$,
so $k( t ) > a - h( 1 ) = 1$. Hence $k^{-1}\bigl( \{ 1 \} \bigr) = \{ 1 \}$.

Furthermore, if $\{ x( 1 ), \ldots, x( n ) \} \subseteq \cL$ then,
by Theorems~\ref{preservers} and~\ref{Krein-Gram},
\begin{align*}
[ k\bigl( [ x( i ), x( j ) ] \bigr) ] & = %
a \bigl[ f[ \bigl( [ x( i ), x( j ) ] \bigr) ] \bigr] - [ h\bigl( [ x( i ), x( j ) ] \bigr) ] \\[1ex]
 & = a \bv \bv^T - a A - [ G\bigl( x( i ), x( j ) \bigr) ] \\[1ex]
 & = \bu \bu^T - B,
\end{align*}
where $\bu := \sqrt{a} \bv$ and $B \in \cS_+$.
The final remark before Section~\ref{examples} gives the claim.
\end{proof}

We note that our class of preservers in $\overline{\cF}$ is closed under this operation:
if $f$ has the form (\ref{preserver}), so that
\[
f( t ) = 1 + \gamma 1_{( 1, \infty )}( t ) + %
\beta \frac{t^\alpha - 1}{\alpha} + \int_{( 0, \infty )} ( 1 - t^{-s} ) \std\nu( s )
\qquad ( t \ge 1 )
\]
with $\beta \ge \alpha ( 1 + \gamma + \nu\bigl( ( 0, \infty ) \bigr) )$, and
\[
h( t ) = \int_{\R_+} t^{-s} \std\mu( s ) = %
-\int_{( 0, \infty )} ( 1 - t^{-s} ) \std\mu( s ) + \mu( \R_+ ) %
\qquad ( t \ge 1 )
\]
for some finite Radon measure $\mu$ on $\R_+$, with $a := 1 + \mu( \R_+ )$,
then
\[
a f( t ) - h( t ) = 1 + a \gamma 1_{( 1, \infty )}( t ) + %
a \beta \frac{t^\alpha - 1}{\alpha} + \int_{( 0, \infty )} ( 1 - t^{-s} ) \std\pi( s )
\qquad ( t \ge 1 ),
\]
where $\pi := a \nu + \mu|_{( 0, \infty )}$, and
\begin{align*}
a \beta & \ge \alpha ( a + a \gamma + a \nu\bigl( ( 0, \infty ) \bigr) ) \\[1ex]
 & = \alpha ( 1 + a \gamma + a \nu\bigl( ( 0, \infty ) \bigr) + \mu( \R_+ ) ) %
\ge \alpha ( 1 + a \gamma + \pi\bigl( ( 0, \infty ) \bigr) ),
\end{align*}
as required.

\subsection{Fractional Ornstein--Uhlenbeck random fields}

In \cite[Section~13]{CL}, Cohen and Lifshits ask whether
the fractional Ornstein--Uhlenbeck random field with covariance
\begin{equation}\label{kernel:OU}
G : ( x, y ) \mapsto \exp\bigl( {-a} \, d( x, y )^\alpha \bigr)
\end{equation}
exists on finite-dimensional real hyperbolic space $\cL_n$ for some $a > 0$ and $\alpha > 1$.
We show here that this is not possible.

\begin{proposition}\label{prp:candl}
Given any $a > 0$ and $\alpha > 1$, there exists an integer $n( a, \alpha )$ such that the fractional
Ornstein--Uhlenbeck field with covariance $G$ as in \tup{(}\ref{kernel:OU}\tup{)}
does not exist in dimension $n$ for any $n \geq n( a, \alpha )$.
\end{proposition}
\begin{proof}
As in (\ref{psdkernel}), we can write $G : ( x, y ) \mapsto h\bigl( [ x, y ] \bigr)$, where here
\[
h : [ 1, \infty ) \to ( 0, \infty ); \ t \mapsto \exp( {-a} ( \arcosh t )^\alpha ).
\]
We suppose for contradiction that $G$ is positive semidefinite.
Given this, it follows from (\ref{psd:rep}) that
\[
F : \R_+ \to ( 0, \infty ); \ x \mapsto h( e^x )
\]
is completely monotone, so $\log \comp F$ is convex \cite[Lemma 4.3]{Merkle}. However,
from (\ref{arcosh}) we have that
\[
\log F( x ) = {-a} \bigl( \arcosh( e^x )\bigr)^\alpha \sim {-a} ( x +\log 2 )^\alpha %
\qquad \text{as } x \to \infty, 
\]
and $x\mapsto {-a} ( x + \log 2 )^\alpha$ is concave. It follows that $G$ is not positive semidefinite.
Thus, for sufficiently large $n$  there exist $x(1)$, \ldots, $x(n) \in \cL$ such that
$\bigl[ G\bigl( x( i ), x( j ) \bigr) \bigr]_{{i, j}=1}^n$ is not positive semidefinite. 
\end{proof}

We conclude this section by noting that a theory of invariant infinitely divisible positive-definite kernels
on symmetric spaces was developed by Gangolli, who obtained a generalised L\'evy--Khintchine representation:
see \cite[Theorem 3.31]{Gangolli}, noting that $\Phi$ in~(3.112) should be $\Psi$.

\section{Screw lines}
The analytic characteristics introduced above have a direct geometric meaning.
To reveal this, we turn to screw lines in Lobachevsky space; these are
curves where the distance between two points depends only on the increment of the curve parameter. Krein obtained a remarkable classification of screw lines
into elliptic, parabolic and hyperbolic types which we recall below.
As every distance preserver transforms a screw line into another screw line,
Krein's trichotomy imposes strong restrictions on the asymptotic behaviour of
preservers.

We recall first that a continuous function $\ell : \R \to \cL$ is a \emph{screw line}
if there exists a function $\lambda : \R \to \R$,
called the \emph{metric function} for the screw line, such that
\[
d\bigl( \ell( x ), \ell( y ) \bigr) = \lambda( x - y ) \qquad \text{for any $x$ and } y \in \R;
\]
equivalently, we have that
\[
[ \ell( x ), \ell( y ) ] = \cosh \lambda( x - y ) \qquad \text{for any $x$ and } y \in \R.
\]
Note that a metric function is necessarily zero at the origin, even and continuous.

The simplest non-trivial example of a screw line in Lobachevsky space is
\[
\ell_0 : \R \to \cL; \ x \mapsto ( \cosh x, \sinh x, 0, 0, \ldots )
\]
since $[ \ell_0( x ), \ell_0( y ) ] = \cosh( x - y )$ and the
corresponding metric function $\lambda_0$ is the identity map.

Krein classified the possible metric functions that may arise in this context: see \cite{Krein} and \cite[Sections~26 and~27]{IK-2}.
The function $\lambda$ arises as the distance function of a screw line $\ell$
if and only if $\Lambda := \cosh \comp \lambda$ has exactly one of three possible
representations:
\begin{alignat*}{2}
& x \mapsto a - \int_{( 0, \infty )} \cos( x s ) \std \sigma( s ) & & %
(\textit{elliptic}); \\[1ex]
& x \mapsto a \cosh( b x ) - \int_{\R_+} \cos( x s ) \std \sigma( s ) & \quad & %
(\textit{hyperbolic}); \\[1ex]
& x \mapsto 1 + \int_{\R_+} \frac{1 - \cos( x s )}{s^2} \std \mu( s ) && (\textit{parabolic}).
\end{alignat*}
The constants $a \in [ 1, \infty )$ and $b \in ( 0, \infty )$,
the positive Radon measure $\sigma$ is supported on $\R_+$,
has total mass $a  - 1$ and no mass at $0$ in the elliptic case,
and the positive Radon measure $\mu$ is supported on $\R_+$ and is such that
\begin{equation}\label{mu}
\int_{( 0, 1 ]} \frac{\rd \mu( s )}{s^2} = \infty \quad \textrm{and} \quad %
\int_{[ 1, \infty )} \frac{\rd \mu( s )}{s^2} < \infty.
\end{equation}

In the elliptic case, the function $\Lambda$ is bounded and therefore so is $\lambda$.
In the hyperbolic case, we have that
\[
\lim_{x \to \infty} \frac{\Lambda( x )}{\cosh(  b x )} = a \qquad %
\textrm{and so} \qquad \lim_{x \to \infty} \frac{\lambda( x )}{x} = b,
\]
since
\[
\frac{\cosh \lambda( x )}{\cosh( b  x )} \sim \exp( \lambda( x ) - b x ) %
\qquad \textrm{as } x \to \infty.
\]
In the parabolic case, we may take any $\delta > 0$ and write
\begin{align*}
x^{-2} \Bigl( \Lambda( x ) - 1 - \frac{x^2}{2} \mu\bigl( \{ 0 \} \bigr) \Bigr) & = %
\int_{( 0, \delta ) \cup [ \delta, 1 ] \cup ( 1, \infty )} %
\frac{1 - \cos( x s )}{x^2 s^2} \std \mu( s ) \\[1ex]
 & \le \mu\bigl( ( 0, \delta ) \bigr) + %
\frac{2 \mu\bigl( [ \delta, 1 ] \bigr)}{x^2 \delta^2} + %
\frac{2}{x^2} \int_{[ 1, \infty )} \frac{\rd\mu( s )}{s^2}.
\end{align*}
Hence $2 x^{-2} \Lambda( x ) \to \mu\bigl( \{ 0 \} \bigr)$ as $x \to \infty$.
The conditions (\ref{mu}) on $\mu$ imply that $\Lambda$
is unbounded (see Proposition~\ref{vNS}). Thus 
$\lambda$ is also unbounded and this working shows that
\[
\lambda( x ) - 2 \log x \to \log \mu\bigl( \{ 0 \} \bigr) \in [ -\infty, \infty ) %
\qquad \text{as } x \to \infty.
\]

\subsection{Consequences for distance preservers}

Suppose $f : [ 1, \infty ) \to [ 1, \infty )$ is a continuous distance preserver
with corresponding metric preserver $\varphi$
and embedding map $\Phi$ from Theorem~\ref{preservers}.

If $\ell$ is a screw line then so is $\Phi \comp \ell$,
with metric function $\varphi \comp \lambda$, since
\[
d( \Phi\bigl( \ell( s ) \bigr), \Phi\bigl( \ell( t ) \bigr) ) = %
\varphi( d\bigl( \ell( s ), \ell( t ) \bigr) ) = \varphi\bigl( \lambda( s - t ) \bigr) %
\qquad \text{ for any $s$ and } t,
\]
and
$\cosh \comp \varphi \comp \lambda = f \comp \cosh \comp \lambda = %
f \comp \Lambda$.

In particular, the screw line $\Phi \comp \ell_0$
has metric function $\varphi \comp \lambda_0 = \varphi$ and
therefore the function $\cosh \comp \varphi = f \comp \cosh$
must have one of the three forms above. Which form appears is controlled
by the characteristics
\[
\alpha := \lim_{t \to \infty} \frac{\log f( t )}{\log t} \in [ 0, 1 ] %
\qquad \text{and} \qquad %
C := \lim_{t \to \infty} t^{-\alpha} f( t ) \in [ 1, \infty ].
\]
If $\alpha = 0$ and $C < \infty$ then $f$ and $\Lambda$ are bounded,
so $\Lambda$ is elliptic. If $\alpha = 0$ and $C = \infty$ then $f$ and $\Lambda$
are unbounded, so $\Lambda$ cannot be elliptic. Nor can $\Lambda$ be hyperbolic,
since
\[
\lim_{x \to \infty} \frac{\Lambda( x )}{\cosh( b x )} = %
\lim_{x \to \infty} 2^{1 - b}\frac{f( \cosh x )}{( \cosh x )^b} = 0 %
\qquad \textrm{for any } b > 0;
\]
this holds because
\[
\lim_{x \to \infty} \frac{( \cosh x )^b}{\cosh( b x )} = 2^{1 - b} %
\quad \text{and} \quad \frac{\log f( t )}{\log t} \le \frac{b}{2} \iff t^{-b} f( t ) \le t^{-b / 2}.
\]
Hence $\Lambda$ is parabolic in this case. Finally, if $\alpha > 0$ then
\[
\frac{\Lambda( x )}{\cosh( \alpha x )} \sim %
2^{1 - \alpha} ( \cosh x )^{-\alpha} f( \cosh x )\to C 2^{1 - \alpha} %
\quad \text{as } x \to \infty,
\]
so $\Lambda$ is hyperbolic, with $a = C 2^{1 - \alpha}$ and $b = \alpha$.
In particular, we see that $C$ must be finite if $f$ is continuous and $\alpha > 0$.
We summarise this working in the following table.
\[
\begin{array}{lll}
\alpha = 0, \ C < \infty & \text{elliptic} & \\[1ex]
\alpha = 0, \ C = \infty & \text{parabolic} & \\[1ex]
\alpha > 0, \ C < \infty & \text{hyperbolic} & a = C 2^{1 - \alpha}, \ b = \alpha \\[1ex]
\alpha > 0, \ C = \infty & \text{impossible}
\end{array}
\]
Thus the analytic dichotomy between zero and positive order and between
finite or infinite type is precisely Krein's geometric trichotomy.

We conclude this section by noting that if $\alpha = 0$ and $C = \infty$
then the representing measure $\mu$ can have no mass at $0$.
To see this, we note that
\begin{equation}\label{arcosh}
\arcosh e^y = \log\bigl( e^y + \sqrt{e^{2 y} - 1} \bigr) \sim %
\log( 2 e^y ) = y + \log 2 \qquad \text{as } y \to \infty,
\end{equation}
so setting $y = \log \cosh x$ gives that
\begin{equation}\label{asymp}
2 y^{-2} f( e^y ) \sim 2 ( \arcosh e^y )^{-2} f( e^y )  = %
2 x^{-2} \Lambda( x ) \to m_0 := \mu\bigl( \{ 0 \} \bigr) %
\qquad \text{as } y \to \infty.
\end{equation}
We let
$\bv := ( e^y, e^{2 y}, e^{3 y} )^T \in ( 1, \infty )^3$
for $y > 0$. Then $\bv \bv^T \in \cM_{1+}^*$ and therefore
\[
f[ \bv \bv^T ] = \bigl[ f( e^{( i + j ) y} ) \bigr]_{i, j = 1}^3 \in \cM_{1+}^*.
\]
If $m_0 > 0$, it follows that
\[
\frac{2}{m_0 y^2} f[ \bv \bv^T ] \to %
\begin{bmatrix} 4 & 9 & 16 \\[1ex] 9 & 16 & 25 \\[1ex] 16 & 25 & 36 \end{bmatrix} %
\qquad \text{as } y \to \infty.
\]
As this matrix has determinant $-8$, it cannot lie in $\cM_{1+}^*$.
Hence we must have $m_0 = 0$.

\subsection{Four examples of screw-line representations}
To illustrate these ideas, we show how the four examples of Section~\ref{examples}
correspond to screw lines and derive their representations.

If $f_1( t ) \equiv t^a$ for some $a \in ( 0, 1 ]$ then
$\Lambda_1 := f_1 \comp \cosh$ is hyperbolic and
\[
( \cosh x )^a = %
2^{1 - a} \cosh( a x ) - \int_{\R} \cos( x s ) \std \sigma_a( s )
\]
for some finite positive Radon measure $\sigma_a$ with total mass
$\sigma_a\bigl( ( 0, \infty ) \bigr) = 2^{1 - a} - 1$. 
Re-arranging, we can write
\[
h( x ) := 2^{1 - a} \cosh( a x ) - ( \cosh x )^a = %
\int_{\R_+} \cos( x s ) \std \sigma_a ( s ),
\]
whence
\begin{align*}
h'( x ) & = %
a 2^{1 - a} \sinh( a x ) - a ( \cosh x )^{a - 1} \sinh x \\[1ex]
\text{and} \quad h''( x ) & = a^2 2^{1 - a} \cosh( a x ) - %
a ( a - 1 ) ( \cosh x )^{a - 2} \sinh^2 x - a ( \cosh x )^a \\[1ex]
 & = a^2 2^{1 - a} \cosh( a x ) - %
a ( a - 1 ) ( \cosh x )^{a - 2} ( \cosh^2 x - 1 ) - a ( \cosh x )^a \\[1ex]
 & = a^2 2^{1 - a} \cosh( a x ) - a^2 ( \cosh x )^a + %
 a ( a - 1 ) ( \cosh x )^{a - 2} \\[1ex]
 & = a^2 h( x ) + a ( a - 1 ) ( \cosh x )^{a - 2}.
\end{align*}
If $q$ is the inverse Fourier transform
of the Schwarz function $x \mapsto ( \sech x )^{2 - a}$ (see Proposition~\ref{schwartz}) then
\[
q( y ) = \frac{1}{2 \pi} \int_{\R} e^{-i x y} ( \sech x )^{2 - a} \std x = %
\frac{1}{2^a \pi} \Beta\Bigl( \frac{2 - a + i y}{2}, \frac{2 - a - i y}{2} \Bigr) \qquad ( y \in \R ),
\]
where $\Beta$ denotes the beta function,
by Proposition~\ref{beta}. Setting
\[
p( y ) := \frac{2 a ( 1 - a )}{a^2 + y^2} q( y ) \qquad \text{for any } y \in \R
\]
and
\[
g( x ) := \int_{\R_+} \cos( x s ) p( s ) \std s = %
\frac{1}{2} \int_{\R} e^{i x y} p( y ) \std y \qquad \text{for any } x \in \R,
\]
we have that
\[
a^2 g( x ) - g''( x ) = a ( 1 - a ) \int_{\R} e^{i x y} q( y ) \std y = %
a ( 1 - a ) ( \sech x )^{2 - a}.
\]
Hence $k := g - h$ is such that $a^2 k - k'' \equiv 0$,
and so there exist constants $A_1$ and $A_2$ such that
\[
g( x ) = %
h( x ) + A_1 \cosh( a x ) + A_2 \sinh( a x ) \qquad \textrm{for any } x \in \R.
\]
Now
\[
0 = g'( 0 ) - h'( 0 ) = a A_2 \implies A_2 = 0
\]
and then
\[
A_1 = \sech( a x ) g( x ) - \sech( a x ) h( x ) \to 0 - 2^{1 - a}  + %
2^{1 - a} = 0 \qquad \text{as } x \to \infty,
\]
since $g$ is a Schwartz function. Thus $g = h$, which shows that
$\sigma_a$ is absolutely continuous with respect to Lebesgue measure and
\[
\frac{\rd\sigma_a}{\rd s} = %
\frac{a ( 1 - a )}{2^{a - 1} \pi ( a^2 + s^2 )} %
\Beta\Bigl( \frac{2 - a + i s}{2}, \frac{2 - a - i s}{2} \Bigr)
\qquad \text{for any } s \in ( 0, \infty ).
\]

If $f_2( t ) \equiv 1 + b \log t$, where $b \in ( 0, \infty )$, then
$\alpha = 0$ and $C = \infty$, so $\Lambda_2 := f_2 \comp \cosh$
is parabolic and we have that
\[
b \log \cosh x= %
\int_{\R_+} \frac{1 - \cos( x s )}{s^2} \std \mu_b( s ) \qquad ( x \ge 0 ),
\]
where $\mu_b = b \mu_1$ is a measure supported on $( 0, \infty )$.

By Proposition~\ref{sech} and the theorems of Fubini and Tonelli, we see that
\begin{align*}
\tanh t = \int_{[ 0, t ]} \sech^2 u \std u &= %
\int_{[ 0, t ]} \int_{\R_+} \cos( s u ) \, s \cosech( \pi s / 2 ) \std s \std u \\[1ex]
 & = \int_{\R_+} \int_{[ 0, t ]} \cos(  s u ) \std u \, s \cosech( \pi s / 2 ) \std s %
\\[1ex]
 & = \int_{\R_+} \frac{\sin( s t )}{s}  \, s \cosech( \pi s / 2 ) \std s.
\end{align*}
Noting that
\[
s \mapsto \frac{\sin( s t )}{s}  \, s \cosech( \pi s / 2 )
\]
extends to a continuous function on $\R_+$ bounded by
$s \mapsto | t | s \cosech( \pi s / 2 )$, which is integrable by Proposition~\ref{cosech},
we have that
\[
\log \cosh x = %
\int_{[ 0, x ]} \int_{\R+} \!\frac{\sin( u s )}{s}  \, s \cosech( \pi s / 2 ) \std s \std u = %
\int_{\R_+} \!\frac{1 - \cos( x s  )}{s^2} \, s \cosech( \pi s / 2 ) \std s.
\]
This shows that $\mu_b$ is absolutely continuous
with respect to Lebesgue measure and has density
\[
\frac{\rd \mu_b}{\rd s} = b s \cosech( \pi s / 2 ) \qquad %
\text{for any } s \in ( 0, \infty ).
\]

For $f_3( t ) \equiv 1 + p - p t^{-a}$, where $p \in ( 0, \infty )$ and
$a \in ( 0, \infty )$, we see immediately that the function
$\Lambda_3 := f_3 \comp \cosh$ is elliptic,
since $f_3$ is bounded, and Proposition~\ref{beta} gives that
\[
f_3( \cosh x ) = 1 + p - \int_{( 0, \infty )} \cos( x s ) \std\tau_a( s ),
\]
where $\tau_a$ is absolutely continuous with density
\[
\frac{\rd \tau_a}{\rd s} = \frac{2^{a - 1} p}{\pi \Gamma( a )} %
\Bigl| \Gamma\Bigl( \frac{a + i s}{2} \Bigr) \Bigr|^2.
\]

For $f_4( t ) \equiv 1 - q + q t$, where $q \geq 1$, we see immediately that
$\Lambda_4 := f_4 \comp \cosh$ is hyperbolic, with
\[
a = \lim_{x \to \infty} \frac{1 - q + q \cosh x}{\cosh x} = q, \qquad b = 1 %
\qquad \text{and} \qquad \sigma = ( q - 1 ) \delta_0.
\]

\section{Characterisation of distance preservers}\label{sec:character}

In this section we will establish the following theorem,
which shows that $\overline{\cF} \setminus \{ 1 \}$
is the whole family of distance preservers.

\begin{theorem}\label{thm:entrywise}
If $f : [ 1, \infty ) \to [ 1, \infty )$ acts entrywise to preserve $\cM_{1+}$ then
$f$ has the representation
\[
f( t ) = 1 + \gamma + \beta \frac{t^\alpha - 1}{\alpha} + \int_{( 0, \infty )} ( 1 - t^{-s} ) \std\nu( s ) %
\qquad ( t > 1 ),
\]
where $\alpha \in [ 0, 1 ]$, $\beta \in \R_+$, $\gamma \in \R_+$ and $\nu \in \cR_2$
are such that $\beta \ge \alpha( 1 + \gamma + \nu\bigl( ( 0, \infty ) \bigr)$,
which is always satisfied if $\alpha = 0$.
\end{theorem}

To prove this result, we first suppose $f$ to be continuous and resolve
separately the positive-order and zero-order cases, and then use the notion of continuous regularisation to settle the general case.

\subsection{The continuous positive-order case}

We claim first that if the continuous distance preserver $f$ 
has order $\alpha \in ( 0, 1 ]$ and type $C \in [ 1, \infty )$ then
the kernel
\[
K : ( x, y ) \mapsto C e^{\alpha ( x + y )} - f( e^{x + y} ) \qquad ( x, y > 0 )
\]
is positive semidefinite.

To see this, we let $x_1$, \ldots, $x_n > 0$; it suffices to show that
\[
\bigl[ C e^{\alpha ( x_i + x_j )} - f( e^{x_i + x_j} ) \bigr]_{i, j = 1}^n = %
D \bigl[ C - e^{-\alpha ( x_i + x_j )} f( e^{x_i + x_j} ) \bigr]_{i, j = 1}^n D
\]
is positive semidefinite,
where $D := \diag( e^{\alpha x_1}, \ldots, e^{\alpha x_n} )$.
We have that
\[
\bigl[ C - e^{-\alpha ( x_i + x_j )} f( e^{x_i + x_j} ) \bigr]_{i, j = 1}^n = %
C \bone_n \bone_n^T - D^{-1} f[ \bv \bv^T ] D^{-1},
\]
where $\bv := ( e^{x_1}, \ldots, e^{x_n} )^T \in ( 1, \infty )^n$. As
$f$ is a distance preserver and $\bv \bv^T \in \cM_{1+}^*$, we see that
\[
M := D^{-1} f[ \bv \bv^T ] D^{-1} \in \cM_{1+}^*,
\]
by Sylvester's law of inertia. Given any $x_0 > 0 $, the same working holds
with $\bv$ replaced by
$\bv_0 = ( e^{x_0}, e^{x_1}, \ldots, e^{x_n} )^T \in ( 1, \infty )^{n + 1}$,
and the corresponding matrix
\[
M_0 = %
\begin{bmatrix} m_{00} & \bm^T \\[1ex] \bm & M \end{bmatrix} \in \cM_{1+}^*,
\]
where
\[
m_{00} := e^{-2 \alpha x_0} f( e^{2 x_0} ) \quad \text{and} \quad %
\bm^T := ( e^{-\alpha ( x_0 + x_i )} f( e^{x_0 + x_i} ) \bigr)_{i = 1}^n;
\]
we note that
\[
m_{00} \to C \quad \text{and} \quad \bm \to C \bone_n %
\qquad \text{as } x_0 \to \infty.
\]
Given any non-zero $\bx \in \R^n$, the matrix $M_0$ acts on the two-dimensional space spanned by $( 1, \bzero )^T$ and $( 0, \bx )^T$ as
\[
M' := %
\begin{bmatrix}
 m_{00} & \bm^T \bx \\[1ex] \bx^T \bm & \bx^T M \bx
\end{bmatrix}.
\]
By the Poincar\'e separation theorem \cite[Corollary~4.3.37]{HJ},
the matrix $M'$ has no more than one positive eigenvalue. It has exactly one because
$m_{00} > 0$ and either $\bx^T M \bx \ge 0$, in which case $M'$ has positive trace,
or $\bx^T M \bx < 0$, in which case $M'$ has negative determinant.
Thus
\[
m_{00} ( \bx^T M \bx ) - ( \bm^T \bx )^2 \le 0 \iff %
\bx^T M \bx \le \frac{( \bm^T \bx )^2}{m_{00}} \to \frac{C^2 ( \bone^T \bx )^2}{C} = %
C \bx^T \bone \bone^T \bx,
\]
whence $\bx^T ( C \bone \bone^T - M ) \bx \ge 0$. This shows that $K$ is
positive semidefinite, as claimed.

We now turn to the following version of a theorem due to Widder
\cite[Theorem~VI.21]{Widder}: if the Hankel kernel
\[
( 0, \infty ) \times ( 0, \infty ) \to \R; \ ( x, y ) \mapsto h( x + y )
\]
is positive semidefinite and $h$ is analytic on $( 0, \infty )$ then there
exists a finite positive Radon measure $\mu$ on $\R$ such that
\[
h( x ) = \int_\R e^{-x s} \std\mu( s ) \qquad ( x > 0 ).
\]
Thus $f$ has the representation
\begin{equation}\label{widder}
f( t ) = C t^\alpha - \int_{\R} t^{-s} \std\mu( s ) \qquad ( t > 1 ),
\end{equation}
where $\mu$ is a finite positive Radon measure on $\R$; by continuity,
this identity also holds if $t = 1$. To conclude,
we show that $\mu$ is supported on $\R_+$.

If $a$, $b \in \R$ are such that $a < b < -\alpha$ and
$\mu\bigl( ( a, b ) \bigr) > 0$ then $b + \alpha < 0$ and
\[
C - t^{-\alpha} f( t ) = \int_{\R} t^{-( s + \alpha )} \std\mu( s ) \ge %
\int_{( a, b )} t^{-( s + \alpha )} \std\mu( s ) \ge %
\mu\bigl( ( a, b ) \bigr) t^{-( \alpha + b )} \to \infty
\]
as $t \to \infty$, which is a contradiction; we recall that $f( t ) = O( t )$ as $t \to \infty$.
Hence $\mu$ has support in $[ -\alpha, \infty )$.
Moreover, we can write
\[
\int_{[ -\alpha, \infty )} t^{-s} \std\mu( s ) = %
\mu\bigl( \{ -\alpha \} \bigr) t^{\alpha} + %
\int_{( -\alpha, \infty )} t^{-s} \std\mu( s ),
\]
so we may assume that $\mu$ has support in $( -\alpha, \infty )$.

Finally, we know that $\Lambda = f \comp \cosh$ is hyperbolic, with $b = \alpha$ and
$a = C 2^{1 - \alpha}$, so there exists a finite positive Radon measure $\sigma$
on $\R_+$ such that
\[
f( \cosh t ) = a \cosh( b t ) - \int_{\R_+} \cos( s t ) \std\sigma( s )
\]
and therefore
\[
I := \int_{( -\alpha, \infty )} ( \cosh t )^{-s} \std\mu( s ) = %
C ( \cosh t )^\alpha - C 2^{1 - \alpha} \cosh( \alpha t ) + %
\int_{\R_+} \cos( s t ) \std\sigma( s )
\]
is bounded for all large $t$. If $\mu\bigl( ( -\alpha, 0 ) \bigr) > 0$ then
$\mu\bigl( ( -\alpha, -\eps ] \bigr) > 0$ for some $\eps \in ( 0, \alpha )$ and
\[
I \ge \mu\bigl( ( -\alpha, -\eps ] \bigr) ( \cosh t )^\eps \to \infty \quad \text{ as } t \to \infty,
\]
which is a contradiction. Hence $\mu$ is supported on $\R_+$ and
so $f \in \cF_1$, which gives the result.

\subsection{The continuous zero-order case}

We now suppose $f$ is continuous and has zero order, and let
\[
\phi : \R_+ \to \R_+; \ x \mapsto f( e^x ).
\]
Then $\phi$ is continuous and non-decreasing on $\R_+$ and
real analytic on $( 0, \infty )$. We also have that $\phi( y ) = o( y^2 )$
as $y \to \infty$: if $f$ is of finite type this is immediate, and it follows from
(\ref{asymp}) otherwise.

We will show first that $-\phi''$ gives rise to a positive-semidefinite Hankel kernel.
For any $\bx := ( x_0, x_1, \ldots, x_n )^T \in ( 0, \infty )^{n + 1}$, we let
\[
\bv := %
( e^{x_0}, e^{x_1}, \ldots, e^{x_n} )^T \in ( 1, \infty )^{n + 1}
\]
and note that
\[
f[ \bv \bv^T ] = \begin{bmatrix}
 \phi( 2 x_0 ) & \bm^T \\[1ex] \bm & M
\end{bmatrix} \in \cM_{1+}^*,
\]
where
$\bm := ( \phi( x_0 + x_1 ), \ldots, \phi( x_0 + x_n ) \bigr)^T$
and
$M = [ \phi( x_i + x_j ) ]_{i, j = 1}^n$.
It follows from the Haynsworth inertia formula \cite{Haynsworth}
that the number of positive eigenvalues of $f[ \bv \bv^T ]$ is one more
than the number of positive eigenvalues of the Schur complement
\[
M - \phi( 2 x_0 )^{-1} \bm \bm^T,
\]
so this last matrix must be negative semidefinite. Hence
\[
\by^T M \by \le %
\frac{1}{\phi( 2 x_0 )} \Bigl( \sum_{i = 1}^n y_i \, \phi( x_0 + x_i ) \Bigr)^2 %
\quad \text{for any } \by \in \R^n.
\]
In particular, if $\by \in \R^n$ is such that $\bone^T \by = 0$
and $x_* := \max\{ x_1, \ldots, x_n \}$
then
\[
\Bigl| \sum_{i = 1}^n y_i \phi( x_0 + x_i ) \Bigr| = %
\Bigl| \sum_{i = 1}^n y_i \bigl( \phi( x_0 + x_i ) - \phi( x_0 ) \bigr) \Bigr|  \le %
\sum_{i = 1}^n | y_i | \, \bigl( \phi( x_0 + x_* ) - \phi( x_0 ) \bigr),
\]
since $\phi$ is non-decreasing, and we can let $x_0 \to \infty$ in such
a way as to make
\[
\frac{\phi( x_0 + x_* ) - \phi( x_0 )}{\sqrt{\phi( 2 x_0 )}} \to 0.
\]
To see this, we suppose for contradiction that $\eps > 0$ and $X > 0$
are such that
\[
\bigl( \phi( x_0 + x_* ) - \phi( x_0 ) \bigr)^2 \ge \eps^2 \phi( 2 x_0 ) %
\qquad \text{for any } x_0 > X.
\]
If $x_0 > \max\{ X, 2 x_* \}$ and
the integer $k \in ( x_0 / ( 2 x_* ) , x_0 / x_* )$ then $2 k x_* > x_0 > k x_*$
and
\begin{align*}
\phi( 2 x_0 ) \ge \phi( 2 x_0 ) - \phi( x_0 ) & \ge %
\sum_{i = 1}^k \phi( x_0 + i x_* ) - \phi( x_0 + ( i - 1 ) x_* ) \\[1ex]
 & \ge \eps \sum_{i = 1}^k \sqrt{\phi\bigl( 2 ( x_0 + ( i - 1 ) x_* ) \bigr)},
\end{align*}
so
\[
\sqrt{\phi( 2 x_0 )} \ge k \eps \ge \frac{x_0 \eps}{2 x_*} \quad \iff \quad
\phi( 2 x_0 ) \ge \frac{\eps^2}{4 x_*^2} x_0^2,
\]
which contradicts the fact that $\phi( y ) = o( y^2 )$ as $y \to \infty$.

This shows that $M$ is conditionally negative definite.
We now suppose $\by^T = ( \bv, {-\bv} )$ for some $\bv^T \in \R^n$
and $\bx = ( x_0, x_1, \ldots, x_m, x_1 + \delta, \ldots, x_m + \delta )$,
where $x_1$, \ldots, $x_m > 0$ and $\delta > 0$. This shows that
the second-order forward-difference matrix
\[
\bigl[ \phi( x_i + x_j ) - 2 \phi( x_i + x_j + \delta ) + %
\phi( x_i + x_j + 2 \delta ) \bigr]_{i, j = 1}^m = %
\bigl[ \Delta_\delta^2 \phi( x_i + x_j ) \bigr]_{i, j = 1}^m
\]
is negative semidefinite,
where the forward increment $( \Delta_\delta k )( t ) := k( t + \delta ) - k( t )$
for any real-valued function $k$ defined on $( 0, \infty )$
and any $\delta > 0$ and $t > 0$,
with $\Delta_\delta^2 := \Delta_\delta \comp \Delta_\delta$.

If $k$ is real analytic on $( 0, \infty )$ then it may be shown by Taylor expansion that
\[
\delta^{-2} \Delta_\delta^2 k( x ) \to k''( x ) \qquad \text{as } \delta \to {0+}.
\]
Hence the matrix
\[
[ \phi''( x_i + x_j ) ]_{i, j = 1}^n
\]
is negative semidefinite for any $x_1$, \ldots, $x_n > 0$, as claimed.

We now use \cite[Theorem~VI.21]{Widder} again to obtain a finite positive Radon
measure on~$\R$ such that
\[
-\phi''( x ) = \int_{\R} e^{-x s} \std\mu( s ) \qquad ( x > 0 ).
\]
We claim that $\mu$ is supported on $[ 0, \infty )$. This representation
gives that $\phi''( x ) \le 0$ for any~$x > 0$ and therefore
$\phi'$ is non-increasing. It is bounded below by zero,
as $f$ is non-decreasing on $( 1, \infty )$ and
$\phi'( x ) = e^x f'( e^x ) \ge 0$, so
\[
-\int_1^\infty \phi''( x ) \std x = \phi'( 1 ) - \lim_{x \to \infty} \phi'( x ) < \infty.
\]
However, if $\delta > 0$ is such that
$c := \mu\bigl( ( {-\infty}, {-\delta} ] \bigr) > 0$ then
\[
-\phi''( x ) \ge c e^{\delta x} \qquad ( x > 0 ) \quad \implies \quad %
-\int_1^\infty \phi''( x ) \std x \ge c \int_1^\infty e^{\delta x} \std x = \infty.
\]
Hence $-\phi''$ is completely monotone, by \cite[Theorem~IV.12a]{Widder},
so $\phi'$ is also completely monotone. Since $f \ge 1$ and $f( 1 ) = 1$,
it follows that the function $B := \phi - 1 \in \cB_0$.
Thus we can write $f ( t ) = 1 + B( \log t )$ and again we have the
representation claimed.

\subsection{The general case}

To conclude, we let $f$ be a distance preserver and recall the continuous regularisation
\[
\tilde{f} : [ 1, \infty ) \to [ 1, \infty ); \ t \mapsto %
\left\{\begin{array}{ll}
f( 1+ ) & \text{if } t = 1, \\[1ex] f( t ) & \text{if } t > 1.
\end{array}\right.
\]
Then $\tilde{f}$ is continuous and acts entrywise to map $\cM_{1+}$
into $\cM_{1+}^*$. Since $f$ is non-decreasing, we see that
$\hat{f} := f( {1+} )^{-1} \tilde{f}$ maps $[ 1, \infty )$ to itself and so
$\hat{f}[ - ]$ sends $\cM_{1+}$ into itself.
Hence there exist $\alpha \in [ 0, 1 ]$, $\hat{\beta} \ge 0$ and
$\hat{\nu} \in \cR_2$ such that
$\hat{\beta} \ge \alpha ( 1 + \hat{\nu}\bigl( ( 0, \infty ) \bigr) )$ and
\[
\hat{f}( t ) = 1 + \hat{\beta} \frac{t^\alpha - 1}{\alpha} + %
\int_{( 0, \infty )} ( 1 - t^{-s} ) \std\hat{\nu}( s ) \qquad ( t \ge 1 ).
\]
Thus
\[
f( t ) = 1 + \gamma 1_{( 1, \infty )}( t ) + \beta \frac{t^\alpha - 1}{\alpha} + %
\int_{( 0, \infty )} ( 1 - t^{-s} ) \std\nu( s ),
\]
where $\gamma := f( {1+} ) - 1$, $\beta := f( {1+} ) \hat{\beta}$
and $\nu := f( {1+} ) \hat{\nu}$, with
\[
\beta = f( {1+} ) \hat{\beta} \ge %
f( {1+} ) \alpha (  1 + \hat{\nu}\bigl( ( 0, \infty ) \bigr) ) = %
\alpha( 1 + \gamma + \nu\bigl( ( 0, \infty ) \bigr) ).
\]
We include the degenerate case where $\alpha$, $\beta$ and $\nu$ are all zero,
as long as $\gamma > 0$.

\section{Multiplicative representation and subordinators}
In this section we introduce a multiplicative representation for continuous
distance preservers and show a connection to the theory of L\'evy processes.

We recall first that if~$f : [ 1, \infty ) \to [ 1, \infty )$ is a distance preserver then
$( \log \comp f )[ - ]$ maps~$\cM_{1+}$ into~$\cN$. If $f$ is also continuous then
it follows \cite[Corollaire~8.2]{FH} that
\begin{equation}\label{cnkernel}
f( t ) = t^\alpha  \exp\Bigl( \int_{( 0, \infty )} \bigl( 1 - t^{-s} \bigr) \std \mu( s ) \Bigr) \qquad %
\textrm{for any } t \in [ 1, \infty ),
\end{equation}
where $\alpha \in \R_+$ and the positive Radon measure $\mu \in \cR_2$.
We call this the \emph{multiplicative} or \emph{exponential Bernstein}
representation of $f$.

We see immediately from (\ref{cnkernel}) that
\[
\frac{\log f( t )}{\log t} = \alpha + %
\int_{( 0, \infty )} \frac{1 - t^{-s}}{\log t} \std\mu( s ) \to \alpha %
\qquad \text{as } t \to \infty,
\]
so our use of $\alpha$ is appropriate:
it is the order of $f$ and so $\alpha \in [ 0, 1 ]$. The type
\[
C = \lim_{t \to \infty} t^{-\alpha} f( t ) = %
\exp\Bigl( \lim_{t \to \infty} \int_{( 0, \infty )} ( 1 - t^{-s} ) \std\mu( s ) \Bigr) = %
e^{\mu\bigl( ( 0, \infty ) \bigr)} \in [ 1, \infty ],
\]
so $f$ has finite type if and only if the measure $\mu$ is finite.
In particular, if $\alpha > 0$ then~$\mu$ must be finite.

\subsection{The relationship between additive and multiplicative representation}

If the continuous distance preserver $f$ has infinite type then $\alpha = 0$
and there exists a Bernstein function $B \in \cB_0$ such that
\[
\int_{( 0, \infty )} ( 1 - e^{-x s} ) \std\mu( s ) = \log f( e^ x ) =: F( x ) = %
\log\bigl( 1 + B( x ) \bigr) \qquad ( x \ge 0 ).
\]
We see that $\mu$ is the representing measure for the Bernstein function
$F$ and so every L\'evy measure $\mu \in \cR_2$ may appear when $\alpha = 0$.

If $f$ has finite type then we have that
\[
C e^{\alpha x} \exp\bigl(  -\int_{( 0, \infty )} e^{-x s} \std\mu( s ) \Bigr)= %
f( e^x ) = C e^{\alpha x} - \int_{\R_+} e^{-x s} \std\nu( s ) \qquad ( x \ge 0 ),
\]
where the finite positive Radon measure $\nu$ on $\R_+$ has
total mass $\nu( \R_+) = C - 1$. Thus, if $L \sigma$ denotes
the Laplace transform of a measure $\sigma$, we see that
\[
( L \nu )( x ) = C e^{\alpha x} ( 1 - e^{-( L \mu )( x )} ) = C e^{\alpha x} %
\sum_{n = 1}^\infty \frac{( {-1} )^{n + 1}}{n!} ( L \mu^{\star n} )( x )
\]
and, equivalently,
\[
( L \mu )( x ) = -\log\bigl( 1 - C^{-1} e^{-\alpha x} ( L \nu )( x ) \bigr) = %
\sum_{n = 1}^\infty \frac{1}{n} ( L \sigma_\alpha^{\star n} )( x ),
\]
where the finite positive Radon measure
\[
\sigma_\alpha :  A \mapsto C^{-1} \nu\bigl( ( A - \alpha ) \cap \R_+ \bigr).
\]
Hence if $f$ has representation (\ref{cnkernel}) and is of finite type
then $\mu$ is supported on $[ \alpha, \infty )$ and the measure
\begin{equation}\label{alternating}
\delta_0 - \exp_\star( -\mu ) = %
\sum_{n = 1}^\infty \frac{( {-1} )^{n + 1}}{n!} \mu^{\star n}
\end{equation}
is positive. Conversely, if the finite Radon measure $\mu$ on $\R_+$
is supported on $[ \alpha, \infty )$ for some $\alpha \in [ 0, 1 ]$ and
(\ref{alternating}) is positive then the function $f$ defined by
(\ref{cnkernel}) is a continuous distance preserver.

To provide an example to illustrate this working, we
suppose $c$, $d \in ( 0, \infty )$ and consider $f$ defined by
(\ref{cnkernel}) with $\alpha \in [ 0, 1 ]$ and $\mu = c \delta_d$, so that
\[
f( e^x ) = e^{\alpha x } \exp\bigl( c ( 1 - e^{-d x} ) \bigr) = %
e^{\alpha x + c} \exp( -c e^{-d x} ) \qquad ( x \ge 0 ).
\]
If $( y_1, y_2, y_3 )^T \in \R_+$ then 
$\bv := ( e^{y_1}, e^{y_2}, e^{y_3} )^T \in [ 1, \infty )^3$ and
$f[ \bv \bv^T ] = D A D$, where
\[
D := e^{c / 2} \diag( e^{\alpha y_1}, e^{\alpha y_2}, e^{\alpha y_3} ) %
\quad \text{and} \quad %
A = \bigl[ \exp\bigl( - c e^{-d ( y_i + y_j )} ) \bigr) \bigr]_{i, j = 1}^3.
\]
Given $h \in ( 0, c / 9 ]$,we set
\[
y_i := -\frac{1}{d} \log\Bigl( i \sqrt{\frac{h}{c}} \, \Bigr) \ge 0 \qquad ( i = 1, 2, 3 ).
\]
Then
\[
-c \exp\bigl( -d ( y_i + y_j ) \bigr) = -i j h
\]
and, letting $x = e^{-h}$,
\[
A = \begin{bmatrix}
 e^{-h} & e^{-2 h} & e^{-3 h} \\[1ex]
 e^{-2 h} & e^{-4 h} & e^{-6 h} \\[1ex]
 e^{-3 h} & e^{-6 h} & e^{-9 h}
\end{bmatrix} = x %
\begin{bmatrix} 1 & x & x^2 \\[1ex] x & x^3 & x^5 \\[1ex] x^2 & x^5 & x^8 \end{bmatrix} %
\not\in \cM_{1+}^*,
\]
since $A$ has determinant $-x^{10} ( 1 - x )^3 ( 1 + x ) < 0$.

It follows that $f$ is not a distance preserver
for any $\alpha \ge 0$ and any $c$, $d > 0$.
It is clear that the positivity constraint on the measure defined by (\ref{alternating})
is violated, as
\[
\sum_{n = 1}^\infty \frac{( {-1} )^{n + 1}}{n!} \mu^{\star n} = %
c \delta_d - \frac{c^2}{2} \delta_{2 d} + \frac{c^3}{6} \delta_{3 d} - \cdots
\]
has negative mass at $2 d$.

\subsection{Four examples of multiplicative representation}

The representation (\ref{cnkernel}) of the function $f_1$ is immediate:
the constant $\alpha = a$ and the measure $\mu$ is the zero measure.

For $f_2$, we note first that
\begin{equation}\label{claimedid}
\log( 1 + x ) = %
\int_{( 0, \infty )} \bigl( 1- e^{-x s} \bigr) \frac{e^{-s}}{s} \std s %
\qquad ( x \geq 0 );
\end{equation}
to see this, we let
\[
H : ( 0, \infty ) \times [ 0, \infty ) \to \R; \ %
( s, x ) \mapsto ( 1 - e^{-x s} ) \frac{e^{-s}}{s}
\]
and note that
\[
\int_{( 0, \infty )} \frac{\partial H}{\partial x}( s, x ) \std s = %
\int_{( 0, \infty )} e^{-( 1 + x ) s} \std s = \frac{1}{1 + x} \qquad ( x \geq 0 )
\]
so (\ref{claimedid}) holds if we can differentiate under the integral sign.
For any $s \in ( 0, \infty )$, the mean-value theorem gives $r \in ( 0, s )$
such that
\[
\frac{1 - e^{-x s}}{s} = x e^{-x r} \in [ 0, x ]
\]
and therefore the right-hand side of (\ref{claimedid}) is well defined.
We also have that
\[
\biggl| \frac{\partial H}{\partial x}( s, x ) \biggr| \leq e^{-s} %
\qquad ( s > 0, \ x \geq 0 )
\]
and differentiation under the integral sign is valid, as this upper bound is an
integrable function of $s$ on $( 0, \infty )$ and is independent of $x$. 

Thus if $b \in ( 0, \infty )$ and $t \in [ 1, \infty )$ then
\begin{equation}
\log f_2( t ) = \log( 1+ b \log t ) = %
\int_{( 0, \infty )} \bigl( 1 - e^{-b s \log t} \bigr) \frac{e^{-s}}{s} \std s = %
 \int_{( 0, \infty)} \bigl( 1 - t^{-r} \bigr) \frac{e^{-r / b}}{r} \std r,
\end{equation}
so the representing measure in (\ref{cnkernel}) is supported on $( 0, \infty )$
and is absolutely continuous there with density
$s \mapsto e^{-s / b} / s$; we note that
\[
\int_{( 0, \infty )} \frac{s}{1 + s} \frac{e^{-s / b}}{s} \std s = %
b \int_{( 0, \infty )} \frac{e^{-u}}{1 + b u} \std u \leq b < \infty.
\]

To find the representation (\ref{cnkernel}) of $f_3$, we note that
\[
\log f_3( t ) = \log( 1 + p ) + \log\Bigl( 1 - \frac{p}{p + 1} t^{-a} \Bigr) = %
-\log\Bigl( 1 - \frac{p}{p + 1} \Bigr) + \log\Bigl( 1 - \frac{p}{p + 1} t^{-a} \Bigr)
\]
and use the fact that
\[
\log( 1 - x ) = -\sum_{n = 1}^\infty \frac{x^n}{n} \qquad ( -1 \le x < 1 )
\]
to write
\[
\log f_3( t ) = %
\sum_{n = 1}^\infty \frac{1}{n} \Bigl( \frac{p}{p + 1} \Bigr)^n ( 1 - t^{-a n} ) = %
\int_{( 0, \infty )} ( 1 - t^{-s} ) \std \mu_{p, a}( s )
\]
where the representing measure
\[
\mu_{p, a} := %
\sum_{n = 1}^\infty \frac{1}{n} \Bigl( \frac{p}{p + 1} \Bigr)^n \delta_{a n}
\]
has total mass $\log( p + 1 )$.

Furthermore, we can set $r := ( q - 1 ) / q $ and use the working for $f_3$ to write
\[
\log f_4( t ) - \log t = \log\Bigl( \frac{1 - q + q t}{t - q t + q t} \Bigr) = %
\log\Bigl( \frac{t - r}{t - t r} \Bigr) = %
\log\Bigl( \frac{1 - ( r / t )}{1 - r} \Bigr) = %
\sum_{n = 1}^\infty \frac{r^n}{n} ( 1 - t^{-n} ).
\]
Hence $f_4 = f_{1, \mu_q}$, where
\[
\mu_q := \sum_{n = 1}^\infty \frac{1}{n} \Bigl( \frac{q - 1}{q} \Bigr)^n \delta_n
\]
has total mass $\log q$.

\subsection{A probabilistic interpretation}\label{sec:probab}

In this section, we follow \cite{Bertoin} but with some changes of notation
to avoid clashes with that used elsewhere.

A \emph{subordinator} is defined to be a non-decreasing L\'evy process,
that is, a stochastic process $S = ( S_u )_{u \in \R_+}$  with non-decreasing
c\`adl\`ag paths and independent stationary increments, such that $S_0 = 0$;
we regard the index $u$ as representing time.

The possibility that $S$ increases to $\infty$ in finite time is included,
so we let the stopping time $\zeta := \inf\{ u \ge 0 : S_u = \infty \}$ and
then, conditionally on $\{ u < \zeta \}$, it holds that the increment
$S_{u + v } - S_u$ is independent of the right-continuous and complete
$\sigma$-algebra generated by $\{ S_w : w \le u \}$ and has the
same distribution as $S_u$.

Given a subordinator $S$, its \emph{Laplace exponent} $F$ is defined
by the identity
\[
\expn[ e^{-x S_u} ] = e^{-u F( x ) } \qquad ( x \ge 0 ),
\]
with the convention that $e^{-x \infty} = 0$ for every $x$.
Then $F$ has the form
\[
F( x ) = k + d x + \int_{( 0, \infty )} ( 1 - e^{-x s} ) \std\Pi( s ) \qquad ( x \ge 0 ),
\]
where the \emph{killing rate} $k \in \R_+$, the drift $d \in \R_+$ and the
L\'evy measure $\Pi \in \cR_2$. Conversely, every such triple gives rise
to the Laplace exponent of some subordinator.

The lifetime $\zeta$ has an exponential distribution with parameter $k$,
so that $\zeta \equiv \infty$ if and only if $k = 0$. Such subordinators are called
\emph{strict}.

An example of a strict subordinator is the Gamma subordinator
$\Gamma = ( \Gamma_u )_{u \ge 0}$ with Laplace exponent $F_\Gamma$
such that
\[
F_\Gamma( x ) := \log( 1 + x ) = %
\int_{( 0, \infty )} ( 1 - e^{-x s} ) \frac{e^{-s}}{s} \std s \qquad ( x \ge 0 ),
\]
where the integral identity follows from (\ref{claimedid}).
A Gamma random variable $X$ with shape~$a$ and rate $b$
has density function
\[
\R_+ \to \R_+; \ x \mapsto \frac{b^a}{\Gamma( a )} x^{a - 1} e^{-b x}
\]
and Laplace transform
\[
\expn[ e^{-x X} ] = \Bigl( \frac{b}{b + x} \Bigr)^a = %
\exp\bigl( -a \log( 1 + b^{-1} x ) \bigr),
\]
so $\Gamma_u$ is a Gamma random variable with shape $a = u$
and rate $b = 1$.

Given any continuous preserver $f$ with multiplicative representation (\ref{cnkernel}),
we see that
\[
F( x ) := \log f( e^x ) = \alpha x + \int_{( 0, \infty )} ( 1 - e^{-x s} ) \std\mu( s )
\]
and there exists a strict subordinator $S$ such that
\[
\expn[ e^{-x S_u} ] = e^{-u F( x )} \qquad \text{ for any } x, u \ge 0.
\]

Henceforth we assume $f$ is a continuous distance preserver with order $\alpha = 0$.
We also have a Laplace exponent from the additive representation of $f$: we let
\[
G( x ) := f( e^x ) = 1 + \beta x + \int_{( 0, \infty )} ( 1 - e^{-x s} ) \std\nu( s ) %
\qquad ( x \ge 0 )
\]
and $H := G - 1 \in \cB_0$. Then there exist subordinators $R$ and $T$,
with $T$ strict, such that
\[
\expn[ e^{-x R_u} ] = e^{-u G( x )} \quad \text{and} \quad %
\expn[ e^{-x T_u} ] = e^{-u H( x )} \qquad \text{ for any } x, u \ge 0.
\]

From the fact that $F = \log G$, we see that if $U_R$ is the renewal measure
of the subordinator $R$ \cite[Section~1.3]{Bertoin} then
\[
\int_{\R_+} e^{-x y} \std U_R( y ) = %
\expn\Bigl[ \int_{\R_+} e^{-x R_u} \std u \Bigr] = %
\int_{\R_+} e^{-u G( x )} \std u = \frac{1}{G( x )} = e^{-F( x )} = \expn[ e^{-x S_1} ].
\]
Hence $U_R$ is equal in distribution to the law of $S_1$.

From the fact that $F = \log( 1 + H ) = F_\Gamma \comp H$, we have the
Bochner subordination \cite[Proposition~8.6]{Bertoin}
\[
S_u \dequals T_{\Gamma_u} \qquad ( u \ge 0 ),
\]
where $\dequals$ denotes equality of distributions and the processes
$\Gamma$ and $T$ are independent.

The continuous distance preserver $f_2( t ) \equiv 1 + b \log t$
has order $\alpha = 0$  and additive representing measure $\nu = 0$,
so $H( x ) \equiv b x$ and $T_u = b u$, whence $S_u \dequals b \Gamma_u$;
note that
\[
\expn[ e^{-x b \Gamma_u}  ] =  \Bigl( \frac{1}{1 + b x} \Bigr)^u = %
\exp\bigl(-u \log( 1 + b x ) \bigr) = \exp\bigl( -u \log f_2( e^x ) \bigr) = %
\expn[ e^{-x S_u} ],
\]
as required.

The continuous distance preserver $f_3( t ) \equiv 1 + p  - p t^{-a}$
has order $\alpha = 0$ and additive representing measure
$\nu_{p, a} = p \delta_a$, so $H( x ) \equiv p ( 1 - e^{-a x } )$
and
\[
\expn[ e^{-x T_u} ] = \exp\big( u p (  e^{-a x} - 1 \bigr) ) \qquad ( u, x \ge 0 ).
\]
If $N_x$ denotes a Poisson random variable with mean $x$ then
we have that $T_u \dequals a N_{p u}$
and therefore $S_u \dequals a N_{p \Gamma_u}$ for all $u \ge 0$,
so $S$ is a negative-binomial process.

\begin{theorem}\label{subord}
Let $S$ be the strict subordinator with Laplace exponent
\[
F : \R_+ \to \R_+; \ x \mapsto \int_{( 0, \infty )} ( 1 - e^{-x s} ) \std\mu( s ) %
\qquad ( \mu \in \cR_2 )
\]
and let $T$ be a subordinator that is independent of $\Gamma$. Then
$S_u \dequals T_{\Gamma_u}$ for all $u \ge 0$ if and only if
$T$ has zero killing rate and Laplace exponent $H$ such that
$F = \log( 1 + H )$.
\end{theorem}
\begin{proof}
If $T$ is a subordinator with Laplace exponent $F_T$ that is independent
of $\Gamma$ then conditioning on $\Gamma_u$ gives that
\[
\expn[ e^{-x T_{\Gamma_u}} ] = \expn[ e^{-\Gamma_u F_T( x )} ] = %
e^{-u ( F_\Gamma \circ F_T )( x )}.
\]
Hence $S_u \dequals T_{\Gamma_u}$ holds for all $u \ge 0$
if and only if $F = F_\Gamma \circ F_T= \log( 1 + F_T )$.

If $T$ has zero killing rate and Laplace exponent $H$ then
$F_T = H$ and so $S_u \dequals T_{\Gamma_u}$ for all $u \ge 0$.

Conversely, if $S_u \dequals T_{\Gamma_u}$ for all $u \ge 0$ then
$F = \log( 1 + F_T )$, so $F_T = e^F - 1$. In particular, we see that
$T$ has zero killing rate, since $F( 0 ) = 0$, and the result follows.

\end{proof}

\begin{corollary}
Suppose $S$, $T$ and $\mu$ are as in the statement of Theorem~\ref{subord}.
If $\mu$ is finite then $S_u \dequals T_{\Gamma_u}$ for all $u \ge 0$
if and only if
\[
\sigma := \delta_0 - \exp_\star( -\mu ) = %
\sum_{n = 1}^\infty \frac{( {-1} )^{n + 1}}{n!} \mu^{\star n}
\]
is a positive measure, in which case the L\'evy measure $\nu$ of $T$ equals
$\exp( \mu\bigl( ( 0, \infty ) \bigr) ) \, \sigma$.
\end{corollary}

\section{Concluding remarks}\label{sec:conclusion}

Schoenberg's classical theorems identify the entrywise transformations compatible
with Euclidean and spherical geometry through positivity and conditional
negativity. The results obtained above complete this picture for spaces of constant curvature by treating the remaining case, that of hyperbolic space. The replacement
of positive-semidefinite Gram matrices by Lorentz--Gram matrices changes the underlying spectral condition from positivity to the requirement of exactly one positive
eigenvalue.

Several classical structures meet in the resulting classification.
Krein's screw lines give geometric meaning to the order and type of a preserver;
Bernstein and L\'evy--Khintchine representations describe its analytic structure;
taking the reciprocal leads to complete Nevanlinna–Pick kernels; the multiplicative
representation for continuous preservers admits an interpretation through
subordinators. These are all different manifestations of the same one-positive-square
phenomenon.

It is natural to ask how much of this picture persists beyond real hyperbolic
space. In particular, one might consider entrywise transformations
for rank-one symmetric spaces, for the fixed-dimensional case, where preservation is
required only up to a prescribed matrix size, and for more general indefinite structures
with a fixed number of positive squares. In each case the central question will be
whether the constraint continues to result in a tractable functional calculus
comparable to the Bernstein-type structure obtained here.

\appendix

\section{Minor proofs and auxiliary results}\label{appendix}

\begin{proposition}\label{decomposition}
Any $x \in \cL \setminus \{ e \}$ may be written uniquely in the form
\[
x = e \cosh \theta + \omega \sinh \theta,
\]
where $\theta \in ( 0, \infty )$ and $\omega \in \ell^2( \N )$ is a unit vector.
Conversely, the vector
\[
x = e \cosh \theta + \omega \sinh \theta \in \cL
\]
for any $\theta \in \R_+$ and any unit vector $\omega \in \ell^2( \N )$. 
\end{proposition}

\begin{proof}
If $x  = ( x_0, x' ) \in \cL \setminus \{ e \}$ then $x' \neq 0$ and
$x_0^2 = 1 + \langle x', x' \rangle > 1$, so $x_0 > 1$.
Hence there exists $\theta \in ( 0, \infty )$ such that $x = e \cosh \theta + \omega \, s$,
where $\omega = s^{-1} \, x'$ and $s > 0$ is such that
\[
s^2 = \langle x', x' \rangle = x_0^2 - 1 = \cosh^2 \theta - 1 = \sinh^2 \theta.
\]
This gives the representation claimed. For uniqueness, if
\[
x = e \cosh \theta_1 + \omega_1 \sinh \theta_1 = %
e \cosh \theta_2 + \omega_2 \sinh \theta_2
\]
then $\cosh \theta_1 = [ e, x ] = \cosh \theta_2$,
so $\theta_1 = \theta_2$ and $\omega_1 = \omega_2$.

For the converse, we have that
\[
[ x, x ] = [ e, e ] \cosh^2 \theta + 2 [ e, \omega ] \cosh \theta \sinh \theta %
 + [ \omega, \omega] \sinh^2 \theta  =  \cosh^2 \theta -  \sinh^2 \theta = %
 1.\qedhere
\]
\end{proof}

\begin{lemma}\label{lem:bound-metric}
Let $( X, d )$ be a metric space and let $\delta \in ( 0, \infty )$. Then
\[
d' : X \times X \to \R_+; \ ( x, y ) \mapsto \min\{ d( x, y ), \delta \}
\]
is also a metric on $X$.
\end{lemma}

\begin{proof}
The only non-trivial thing to check is that the triangle inequality holds.
For this, let $x$, $y$, $z \in X$ and
note that if $d( x, y ) \geq \delta$ or $d( y, z ) \geq \delta$ then
\[
d'( x, y ) + d'( y, z ) \geq \delta \geq d'( x, z ).
\]
Finally, if $d( x, y ) < \delta$ and $d( y, z ) < \delta$ then
\[
d'( x, y ) + d'( y, z ) = d( x, y ) + d( y, z ) \geq d( x, z ) \geq d'( x, z ).\qedhere
\]
\end{proof}

\begin{proposition}\label{counterexample}
Fix $\delta \in ( 0, \infty )$ and let
\[
\vartheta : \R \to \R; \ x \mapsto \min\{ x, \delta \}.
\]
Then  $\vartheta^{-1}\bigl( \{ 0 \} \bigr) = \{ 0 \}$ and $\vartheta \comp d$
is a metric on $\cL$, where $d$ is the hyperbolic metric,
but the kernel corresponding to
\[
f : [ 1, \infty ) \to [ 1, \infty ); \ t \mapsto \cosh \vartheta( \arcosh t ) = %
\min\{ t, \cosh \delta \}
\]
does not have exactly one positive square.
\end{proposition}

\begin{proof}
It is immediate that $\vartheta^{-1}\bigl( \{ 0 \} \bigr) = \{ 0 \}$, and
$\vartheta \circ d$ is a metric by Lemma~\ref{lem:bound-metric}.
Furthermore, if $h := \delta / 2$ and
\[
x( i ) := ( \cosh( i h ), \sinh( i h ), 0, 0, \ldots ) \in \cL \qquad ( i = 1, 2, 3, 4)
\]
then $[ x( i ), x( j ) ] = \cosh\bigl( ( i - j ) h )$.
Hence if $M := \bigl[ [ x( i ), x( j ) ] \bigr]_{i, j = 1}^4 \in \cM_{1+}$, $a := \cosh h$
and $b := \cosh \delta = \cosh( 2 h ) = 2 a^2 - 1$ then
\[
f[ M ] = \begin{bmatrix}
1 & a & b & b \\[1ex] a & 1 & a & b \\[1ex] b & a & 1 & a \\[1ex] b & b & a & 1
\end{bmatrix}
\not\in \cM_{1+},
\]
since
\[
\bu^T f[ M ] \bu = 4 + 6 a + 6b > 0 \quad \text{and} \quad %
\bv^T f[ M ] \bv = 4 - 6 a + 2 b = 2 ( a - 1 ) ( 2 a - 1 ) > 0
\]
for orthogonal vectors $\bu := ( 1, 1, 1, 1 )^T$ and $\bv := ( 1, {-1}, 1, {-1} )^T$
with $\bu^T f[ M ] \bv = 0$.
\end{proof}

\begin{proposition}\label{examplematrix}
Given any $a$, $b \in \R$, the matrix
\[
A = M_{a, b} := \begin{bmatrix} 1 & a & b \\ a & 1 & a \\ b & a & 1 \end{bmatrix}
\]
has eigenvalues
\[
1 - b, \qquad 1 + \frac{1}{2} \bigl( b - \sqrt{8 a^2 + b^2} \bigr) \quad \text{and} \quad %
1 + \frac{1}{2} \bigl( b + \sqrt{8 a^2 + b^2} \bigr).
\]
If $a = b$ then $A$ has eigenvalues $1 - a$ and $2 a + 1$ with multiplicities $2$ and $1$ respectively
and if $A \in \cM$ then $A$ has exactly one positive eigenvalue if and only if $b \leq 2 a^2 - 1$.
\end{proposition}
\begin{proof}
It is readily verified that
\[
\det( \lambda I - A ) = ( \lambda - 1 + b ) ( \lambda^2 - ( b + 2 ) \lambda - 2 a^2 + b + 1 ).
\]
The claims follow in a routine manner.
\end{proof}

\begin{proposition}\label{example2matrix}
Given any $a$, $b$, $c \in \R$ and $m \in \N \cap [ 2, \infty )$, the $2 m \times 2 m$ block matrix
\[
A = N_{a, b, c}^{(m)} := \begin{bmatrix}
B_a  & b \bone_m \bone_m^T \\[1ex] b \bone_m \bone_m^T & B_c
\end{bmatrix},
\]
where $B_x := ( 1 - x ) \Id_{m \times m} + x \bone_m \bone_m^T$ for any $x \in \R$,
has eigenvalues $1 - a$ and $1 - c$ of total multiplicity $2 m - 2$,
and
\[
\frac{1}{2} \bigl( 2 + ( m - 1 ) ( a + c ) \pm \sqrt{ ( m - 1 )^2 ( a - c )^2 + 4 m^2 b^2} \bigr).
\]
Thus $A \in \cM^*$ has exactly one positive eigenvalue if and only if
\[
m^2 b^2 \ge ( 1 + ( m - 1 ) a ) ( 1 + ( m - 1 ) c ).
 \]
\end{proposition}
\begin{proof}
If $\bv \in \R^m$ is orthogonal to $\bone_m$
then $( \bv \oplus \bzero_m )^T$ and $( \bzero_m \oplus \bv )^T$ are eigenvectors of $A$
with eigenvalues $1 - a$ and $1 - c$, respectively. This proves the first claim.

On the space spanned by $( \bone_m \oplus \bzero_m )^T$ and $( \bzero_m \oplus \bone_m )^T$,
the matrix $A$ acts as
\[
\begin{bmatrix} 1 - a + m a & m b \\ m b & 1 - c + m c \end{bmatrix}.
\]
The result follows.
\end{proof}

\begin{proposition}\label{schwartz}
For any $b \in ( 0, \infty )$,
the function $x \mapsto ( \sech x )^b$ is a Schwartz function on $\R$.
\end{proposition}
\begin{proof}
We recall that
\[
\frac{\rd}{\rd x} \sech x = -\sech x \tanh x \qquad \text{and} \qquad %
\frac{\rd}{\rd x} \tanh x = \sech^2 x \qquad \text{for any } x \in \R.
\]
Furthermore, we have that
\[
| \sech x | \leq 2 e^{-| x |} \qquad \text{and} \qquad | \tanh x | \leq 1 %
\qquad \text{for any } x \in \R.
\]
We claim that the $n$th derivative of $x \mapsto ( \sech x )^b$
is a finite linear combination of terms of the form
$( \sech x )^c ( \tanh x )^m$, where $c \geq b$ and $m \in \Z_+$.
This is true for $n = 1$, so holds by induction, as
\[
\frac{\rd}{\rd x} ( \sech x )^c ( \tanh x )^m = %
-c ( \sech x )^c ( \tanh x )^{m + 1} + m ( \sech x )^{c + 2} ( \tanh x )^{m - 1}.
\]
The result follows.
\end{proof}

\begin{proposition}\label{cosech}
For any $b \in ( 0, \infty )$ and any $n \in \N$,
the function $x \mapsto x^n \cosech( b x )$ is integrable on $\R$.
\end{proposition}
\begin{proof}
Note first that
\[
\lim_{x \to 0} \frac{x}{\sinh( b x )} = %
\lim_{x \to 0} \frac{1}{b \cosh( b x )} = \frac{1}{b},
\]
so the function $x \mapsto x^n \cosech( b x )$ can be made continuous
at the origin. Furthermore, if $x \neq 0$ then
\[
\Bigl| \frac{x^n}{\sinh( b x )} \Bigr| \leq %
\frac{( 2 n + 1 )! | x |^n}{b^{2 n + 1} | x |^{2 n + 1}} = %
\frac{( 2 n + 1 )!}{b^{2 n + 1}} \, | x |^{{-n} - 1}
\]
and the result follows.
\end{proof}

\begin{proposition}\label{sech}
If $x \in \R$  and $b \in ( 0, \infty )$ then
\[
\int_{\R_+} \cos( x s ) \sech( b s ) \std s = %
\frac{\pi}{2 b} \sech( \pi x / 2 b ),
\]
so
\begin{align*}
\sech x & = \int_{\R_+} \cos( x s ) \sech( \pi s / 2 ) \std s \\[1ex]
\textrm{and} \quad \sech^2 x & = %
\int_{\R_+} \cos( x s ) \, s \cosech( \pi s / 2 ) \std s.
\end{align*}
\end{proposition}
\begin{proof}
It is readily verified that $\cosh$ has a simple zero at $( 2 n + 1 ) \pi i / 2$ for any $n \in \Z$
and is non-zero otherwise. Furthermore, if $x \in \R$ then
\[
\lim_{z \to \pi i / 2} \frac{( z - ( \pi i / 2 ) \bigr) e^{i x z}}{\cosh z} = %
\frac{e^{{-x} \pi / 2}}{\sinh( \pi i / 2 )} = {-i} e^{{-\pi} x / 2}.
\]
Hence, by Cauchy's residue theorem, if $N \in \N$ then
\[
\int_{[ {-N}, N ] \cup [ N, N + \pi i ] \cup [ N + \pi i, {-N} + \pi i ] \cup [ {-N} + \pi i, {-N} ]} %
\frac{e^{i x z}}{\cosh z} \std z = 2 \pi e^{{-\pi} x / 2},
\]
where $[ z, w ]$ denote the directed line segment from $z$ to $w$.

If $z = N + i y \in [ N, N + \pi i ]$ then
\[
\Bigl| \frac{e^{i x z}}{e^z + e^{-z}} \Bigr| = \frac{e^{{-x} y} e^N}{| e^{2 N + 2 i y} + 1 |} \leq %
\frac{e^N}{e^{2 N} - 1}
\]
and similarly if $z = {-N} + i y \in [ -N - \pi i, -N ]$ then
\[
\Bigl| \frac{e^{i x z}}{e^z + e^{-z}} \Bigr| = \frac{e^{{-x} y} e^N}{| 1 + e^{2 N - 2 i y} |} \leq %
\frac{e^N}{e^{2 N} - 1},
\]
so
\[
\Bigl| \int_{[ N, N + \pi i ] \cup [ {-N} - \pi i, {-N} ]} \frac{e^{i x z}}{\cosh z}  \std z \Bigr| \leq %
\frac{4 \pi e^N}{e^{2 N} - 1} \to 0 \quad \text{as } N \to \infty.
\]
Furthermore, we have that
\[
\int_{[ {-N}, N ]} \frac{e^{i x z}}{\cosh z} \std z = %
\int_{[ {-N}, N ]} \frac{e^{i x u}}{\cosh u} \std u = %
2 \int_{[ 0, N ]} \cos( x t ) \sech t \std t
\]
and
\[
\int_{[ {-N} + \pi i, N + \pi i ]} \frac{e^{i x z}}{\cosh z} \std z = %
\int_{[ {-N}, N ]} \frac{e^{i x u} e^{-\pi x}}{\cosh( u + \pi i )} \std u = %
-2 e^{-\pi x} \int_{[ 0, N ]} \cos( x t ) \sech t \std t,
\]
since
\[
\cosh( u + \pi i ) = \frac{e^u e^{\pi i} + e^{-u} e^{{-\pi} i}}{2} = {-\cosh u}.
\]
Rearranging and letting $N \to \infty$, we have that
\[
\int_{\R_+} \cos( x t ) \sech t \std t = %
\frac{2 \pi e^{-\pi x / 2}}{2 ( 1 + e^{-\pi x} )} = \frac{\pi}{2 \cosh( \pi x / 2 )}
\]
and hence
\[
\int_{\R_+} \cos( x s ) \sech( b s ) \std s = %
\frac{1}{b} \int_{\R_+} \cos( x r / b ) \sech( r ) \std r = %
\frac{\pi}{2 b} \sech( \pi x / 2 b ),
\]
as claimed.

The second identity follows immediately. For the third, note that the second gives that
\[
\sech^2 x = \frac{1}{4} %
\int_\R \int_\R e^{i x v} \sech( \pi u / 2 ) \sech( \pi ( v - u ) / 2 ) \std u \std v.
\]
Now,
\begin{alignat*}{2}
\frac{1}{4} \int_\R \sech( \pi u / 2 ) & \sech( \pi ( v - u ) / 2 ) \std u \\[1ex]
 & = \int_\R \frac{1}{e^{\pi u / 2} + e^{{-\pi} u / 2}}%
\frac{1}{e^{\pi ( v - u ) / 2} + e^{\pi ( u - v ) / 2}} \std u \\[1ex]
 & = \int_{\R_+} \frac{1}{t + t^{-1}}\frac{1}{e^{\pi v / 2} t^{-1} + %
 e^{-\pi v / 2} t} \frac{2 \std t}{\pi t} && ( t = e^{\pi u / 2} ) \\[1ex]
 & = \frac{2 b}{\pi} \int_{\R_+} \frac{t}{( t^2 + 1 ) ( t^2 + b^2 )} \std t %
 && ( b = e^{\pi v / 2} ) \\[1ex]
 & = \frac{2 b}{\pi ( b^2 - 1 )} %
 \int_{\R_+} \frac{t}{t^2 + 1} - \frac{t}{t^2 + b^2} \std t %
 & \qquad & ( b \neq 1 ) \\[1ex]
 & = \frac{1}{2 \pi \sinh( \pi v / 2 )} %
 \biggl[ \log\Bigl( \frac{t^2 + 1}{t^2 + b^2} \Bigr) \biggr]_0^\infty \\[1ex]
 & = \frac{v}{2 \sinh( \pi v / 2 )},
 \end{alignat*}
which is an even function of $v$, and therefore
\[
\sech^2 x = \int_{\R_+} \cos( x s ) \, s \cosech( \pi s / 2 ) \std s.\qedhere
\]
\end{proof}

\begin{proposition}\label{beta}
If $x \in \R$ and $b \in ( 0, \infty )$ then
\[
\int_{\R_+} \cos( x t ) ( \sech t )^b \std t = %
2^{b - 2} \Beta\Bigl( \frac{b + i x}{2}, \frac{b - i x}{2} \Bigr) = %
\frac{2^{b - 2}}{\Gamma( b )} %
\Bigl| \Gamma\Bigl( \frac{b + i x}{2} \Bigr) \Bigr|^2 \geq 0,
\]
where the beta and gamma functions
\[
\Beta( z, w ) := \int_{[ 0, 1 ]} t^{z - 1} ( 1 - t )^{w - 1} \std t = %
\frac{\Gamma( z ) \Gamma( w )}{\Gamma( z + w )} \qquad \text{and} \qquad %
\Gamma( z ) := \int_{\R_+} t^{z - 1} e^{-t} \std t 
\]
for any $z$, $w \in \C$ such that $\Re z > 0$ and $\Re w > 0$.
\end{proposition}
\begin{proof}
The substitution $t = ( 1 + s )^{-1}$ gives that
\[
\int_{[ 0, 1 ]} t^{z - 1} ( 1 - t )^{w - 1} \std t = %
\int_{\R_+} ( 1 + s )^{-z - w} s^{w - 1} \std s.
\]
Hence
\begin{align*}
\Beta\Bigl( \frac{b + i x}{2}, \frac{b - i x}{2} \Bigr) & = %
\int_{\R_+} ( 1 + s )^{-b} s^{( b - i x ) / 2} \, \frac{\rd s}{s} \\[1ex]
 & = \int_{\R_+} ( s^{-1 / 2} + s^{1 / 2} )^{-b} s^{-i x / 2} \, \frac{\rd s}{s} \\[1ex]
 & = 2 \int_{\R} ( e^y + e^{-y} )^{-b} e^{i x y} \std y %
 \hspace{6em} ( s = e^{-2 y} ) \\[1ex]
 & = 2^{2 - b} \int_{\R_+} \cos( x t ) ( \sech t )^b \std t.
\end{align*}
The final claim now follows from the well known identity linking the beta and gamma functions
and the fact that $\overline{\Gamma( z )} = \Gamma( \overline{z} )$ for any $z \in \C$ 
with $\Re z > 0$.
\end{proof}

\begin{proposition}\label{vNS}
Let $\mu$ be a positive Radon measure on $( 0, \infty )$. The function
\[
f : \R \to \R_+; \ t \mapsto \int_{( 0, \infty )} \frac{1 - \cos( s t )}{s^2} \std\mu( s )
\]
is bounded if and only if the integral
\[
I := \int_{( 0, \infty )} \frac{\rd\mu( s )}{s^2}
\]
is finite.
\end{proposition}
\begin{proof}
We follow the proof of \cite[Theorem~3]{vNS}. 

Since $s \mapsto \bigl( 1 - \cos( s t ) \bigr) / s^2$ has a limit as $s \to 0+$ and
\[
0 \leq \frac{1 - \cos( s t )}{s^2} \leq \frac{2}{s^2},
\]
if $I$ is finite then $f$ is bounded, with $0 \leq f( t ) \leq 2 I$ for all $t \in \R$.

Conversely, suppose $M > 0$ is such that $0 \leq f( t ) \leq M$ for all $t \in \R$ and let
\[
f_{\epsilon, a}( t ) := \int_{[ \epsilon, a ]} \frac{1 - \cos( s t )}{s^2} \std\mu( s ) %
\qquad ( 0 < \epsilon < a < \infty ).
\]
Then, by Tonelli's theorem and the dominated convergence theorem, we have that
\[
b_n := \frac{1}{n} \int_{[ 0, n ]} f_{\epsilon, a}( t ) \std t = %
\int_{[ \epsilon, a ]} \Bigl( 1 - \frac{\sin( n s )}{n s} \Bigr) \frac{\rd\mu( s )}{s^2} %
\to \int_{[ \epsilon, a ]} \frac{\rd\mu( s )}{s^2}
\]
as $n \to \infty$. Furthermore, we know that
\[
0 \leq f_{\epsilon, a}( t ) \leq f( t ) \leq M \qquad \text{for any }  t \in \R,
\]
so $0 \leq b_n \leq M$ for all $n$.
Hence $0 \leq \int_{[ \epsilon, a ]} \rd\mu( s ) / s^2 \leq M$
and, as $M$ is independent of $\epsilon$ and $a$, it follows that $I$ is finite, with $0 \leq I \leq M$.
\end{proof}



\end{document}